\documentclass[12 pt]{amsart}

\usepackage[pdftex,
colorlinks, 
bookmarksnumbered, 
bookmarksopen, 
linkcolor=blue, 
citecolor=blue, 
]{hyperref}
\usepackage{graphicx}
\usepackage{amssymb}
\usepackage{amsthm}
\providecommand {\norm}[1] {\lVert#1\rVert}
\providecommand {\bignorm}[1] {\Bigl\lVert#1\Bigr\rVert}

\providecommand {\abs}[1] {\lvert#1\rvert}
\providecommand {\bigabs}[1] {\Bigl\lvert#1\Bigr\rvert}
\providecommand {\inprod}[1]{\langle #1 \rangle}
\providecommand {\biginprod}[1]{\Big\langle #1 \Big\rangle}
\providecommand {\set}[1]{\lbrace #1 \rbrace}

\providecommand {\inv}[1]{{#1}^{-1}}

\newcommand {\one} e
\newcommand {\bn} {\ensuremath{\mathbb{N}}}

\newcommand {\br} {\ensuremath{\mathbb{R}}}
\newcommand {\brp}{\ensuremath{\mathbb{R}_0^+}}
\newcommand {\bz} {\ensuremath{\mathbb{Z}}}

\newcommand {\bc} {\ensuremath{\mathbb{C}}}

\newcommand {\mD} {\ensuremath{\mathcal{D}}}

\newcommand {\mF} {\ensuremath{\mathcal{F}}}

{

\newcommand {\mO} {\ensuremath{\mathcal{O}}}

\newcommand {\pp} {\ensuremath{\sqrt{p_+}}}
\newcommand {\pmi} {\ensuremath{\sqrt{p_-}}}

\newcommand {\Ltwo} {L^2(\br)}
\newcommand {\PW} [2] [{A_p}] {PW_{#2}(#1)}

\DeclareMathOperator{\supp}{supp}
\DeclareMathOperator{\spann}{span}

\DeclareMathOperator{\sinc}{sinc}

\newcommand \BL {bandlimited}

\newcommand \stft {short time Fourier transform}

\newcommand \cont {continuous}

\newcommand \saj {self-adjoint} 

\newcommand \SL {Sturm Liouville}

\newcommand \ndc {necessary density conditions}

\newtheorem{prop}{Proposition} [section]
\newtheorem{cor}[prop]{Corollary}
\newtheorem{thm}[prop]{Theorem}

\newtheorem{lem}[prop]{Lemma}

\theoremstyle{definition}
\newtheorem{defn}[prop]{Definition} 

\theoremstyle{remark}
\newtheorem*{rem}{Remark}
\newtheorem*{rems}{Remarks}

\numberwithin{equation}{section}

\setlength{\topmargin}{-8mm}
\setlength{\headheight}{8pt}
\setlength{\textheight}{220mm}  

\setlength{\oddsidemargin}{0pt}
\setlength{\evensidemargin}{0pt}
\setlength{\textwidth}{148 mm}   
\newcommand{\up}{uncertainty principle}
\newcommand{\tfa}{time-frequency analysis}

\newcommand{\ft}{Fourier transform}

\newcommand{\tf}{time-frequency}

\newcommand{\fif}{if and only if}

\newcommand{\bdl}{band\-limited}

\newcommand{\bw}{bandwidth}

\begin{document}
\begin{abstract}
We propose a new notion of variable bandwidth that is based on the
 spectral subspaces of  an elliptic
operator $A_pf =  - \tfrac{d}{dx} (p(x) \tfrac{d}{dx})f$ where $p>0$ is a
strictly
positive function. Denote by  $c_{\Lambda}  (A_p)$ the orthogonal
projection of $A_p$ corresponding to the spectrum of $A_p$ in $\Lambda
\subset \br ^+$, the range of this projection is the space of
functions of variable \bw\ with spectral set in  $\Lambda $.

We will develop the basic theory of these function
spaces. First, we  derive (nonuniform) sampling theorems, second, we
prove  necessary density conditions in the style of Landau. Roughly,
for a  spectrum  $\Lambda = [0,\Omega] $    the main
results say   that, in a neighborhood of $x\in \br$,   a function of
variable \bw\  behaves like a bandlimited function with local
bandwidth $(\Omega / p(x))^{1/2}$.

Although  the formulation of
the results is deceptively  similar to the corresponding results for
classical bandlimited
functions, the methods of proof are much more
involved. On the one hand, we use the oscillation method from sampling
theory and frame
theoretic methods, on the other hand, we need the
precise spectral theory of Sturm-Liouville operators and the
scattering theory of one-dimensional Schr\"odinger operators.
\end{abstract}

\title[Variable Bandwidth]{What is Variable Bandwidth?}


\author{Karlheinz Gr\"ochenig},
\author{ Andreas Klotz}
\address{Faculty of Mathematics \\
University of Vienna \\
Oskar-Morgenstern-Platz 1 \\
A-1090 Vienna, Austria}
\email{karlheinz.groechenig@univie.ac.at}
\email{andreas.klotz@univie.ac.at}

\subjclass[2010]{}
\date{\today}
\thanks{This work  was 
  supported  by the  project P26273 - N25  of the
Austrian Science Fund (FWF)}
\maketitle
\section{Introduction}


 A function   $f\in L^2(\br )$  has \bw\ $\Omega >0$, if its
Fourier transform $\hat{f} (\xi )  = \int _{\br} f(x) e^{-ix\xi } \,
dx $ vanishes outside the interval $[-\Omega , \Omega ]$. The number
$\Omega  $ is the maximal frequency contributing to $f$ and is called
the \bw\ of $f$. According to Shannon
the \bw\ is an information-theoretic quantity and determines  how
many samples of a function $f$ are required to determine $f$
completely. Alternatively, the \bw\  indicates how much information
can be transmitted through a communication channel.

 In the context of \tfa\ it is perfectly plausible to 
assign different   local \bw s to different  segments of a
signal. This becomes even more obvious in the often cited
metaphor of music: the highest frequency of musical piece is
time-varying. However, a rigorous definition of variable \bw\ is
difficult, perhaps even elusive, because \bw\ is global by definition
and the assignment of a local \bw\ is in contradiction with the \up . 

\vspace{2 mm}

\begin{center}
 So what is a function of variable \bw ?  
\end{center}

\vspace{2 mm}



Before attempting to give a precise definition, we need to single out the
distinctive features of \bdl\ functions. In our view the essence of
\bw\ is encapsulated in three fundamental theorems about \bdl\
functions: 
\begin{enumerate}
\item the Shannon-Whittaker-Kotelnikov sampling theorem and its variations,
\item the existence of a critical density in the style of
  Landau's necessary density conditions (a Nyquist rate in engineering
  terms), and
\item some inherent analyticity, expressed by a Bernstein-type
  inequality and theorems in the style of Paley-Wiener. 
\end{enumerate}

Our objective is to introduce a concept of variable \bw\ that shares
these fundamental properties (sampling theorems and density results)
with classical \bdl\ functions. Our starting point is the well known
observation that \bdl\ functions are contained in a spectral subspace 
of the 
differential operator $-D^2 = -\tfrac{d^2}{dx^2}$. 
It is diagonalized by the
Fourier transform  $\mF $ and $\mF   D^2  \mF ^{-1} f =  \xi ^2
\hat{f}(\xi ) $ is simply the  operator of multiplication by 
$ \xi ^2$. For fixed  $\Omega >0$ the spectral subspace
corresponding to the spectral values $0\leq  \xi ^2 \leq \Omega $
is defined by  the spectral projection
$c_{[0,\Omega ]} (-D^2)$, which is given by $\hat{f} \mapsto
c_{[0,\Omega ]} ( \xi ^2) \hat{f}(\xi )$ in the Fourier
domain. Consequently the spectral subspace $c_{[0,\Omega ]} (-
D^2)L^2(\br )$ 
is identical with the \bdl\ functions of \bw\ $
\Omega  ^{1/2}$. 

Our idea is to  replace the constant coefficient differential operator
$- D^2$ by an elliptic
differential operator in divergence form  
\begin{equation}
  \label{eq:c1}
  A_p = - D (p (x) D) \, , 
\end{equation} 
with a \emph{\bw -parametrizing function} $p>0$. As above, a  space of variable
\bw\ is given by a spectral subspace of the differential operator $A_p$.
By imposing mild assumptions on $p$ and  choosing  a suitable domain,
$A_p$ becomes a positive,  unbounded, 
self-adjoint operator on $L^2(\br)$. Its  spectral
representation enables us to   make the following definition.

\begin{defn} \label{def1}
Let $\Lambda \subseteq \br ^+$ be a fixed  Borel set with finite Lebesgue measure.   A
function is $A_p$-\bdl\ 
with spectral set   $\Lambda $, if $f\in
  c_{\Lambda } (A_p) L^2(\br) $. The range of the spectral projection $
  c_{\Lambda } (A_p) L^2(\br) $ is called the  Paley-Wiener space with respect to
  $A_p$ and  spectral set  $\Lambda $ and will be 
  denoted by $PW_\Lambda (A_p)$.  The quantity 
$\Omega = \max \{ \lambda \in \Lambda \}$ is the \bw\ of $PW_\Lambda
(A_p)$.  If $\Lambda = [0,\Omega ]$, we will often speak of functions
of variable \bw\ $\Omega $. 
\end{defn}

If $p \equiv 1$ and $A_p = -\tfrac{d^2}{dx^2}$, then, as argued above,
$PW_{[0,\Omega
  ]}(A_p)$ consists exactly of the classical \bdl\ functions with
\bw\ ${\Omega }^{1/2}$. 


 Our challenge is to convince the reader that Definition~\ref{def1}
 is indeed a meaningful  notion of variable \bw . 
 We must 
 interpret functions in $PW_\Lambda (A_p)$  as functions of variable
 \bw\ and need to relate the parametrizing function $p$ to a local \bw
 . Furthermore, we need to establish sampling and density theorems for
 $PW_\Lambda (A_p)$ that depend  to the  \bw -parametrizing
 function $p$.   
As a guideline,  we would
 expect that $1/\sqrt{p(x)}$ determines the local \bw\ in a
 neighborhood of $x$ and will enter the formulation of the basic
 results. 

First, we show that functions of variable \bw\ admit sampling theorems.

\begin{thm}[Sampling theorem for $PW_\Lambda (A_p)$] \label{tm1}
Fix $\Lambda \subseteq \br ^+$ compact  and set $\Omega = \max
 \Lambda $.   Assume that $0< c \leq p(x) $ for all $x
  \in \br $.  Let $X = (x_i) _{i\in \bz }$
  be an increasing sequence with $\lim _{i \to \pm \infty } x_i = \pm
  \infty $ and $\inf _i (x_{i+1} - x_i ) > 0$. If
  \begin{equation}
    \label{eq:c2}
    \delta = \sup _{i\in \bz} \frac{x_{i+1}-x_i}{\inf _{x\in
        [x_i,x_{i+1}]} \sqrt{p(x)}} < \frac{\pi }{{{\Omega }^{1/2}}} \, ,
  \end{equation}
then there exist $A,B>0$ such that,   for all $ f \in PW_\Lambda
(A_p)$,  
\begin{equation}
  \label{eq:c3}
  A \|f \|_2^2 \leq \sum _{i\in \bz} |f(x_i)|^2 \leq B \|f\|_2^2 \qquad
\end{equation}
\end{thm}

Since $PW_\Lambda (A_p)$ is a reproducing kernel Hilbert space, the
sampling inequality~\eqref{eq:c3} implies a variety of reconstruction
algorithms. Following~\cite{Groch92}  we will formulate an iterative  algorithm for the
reconstruction of $f\in PW_\Lambda (A_p)$ from the samples $\set{f(x_i) \colon 
i \in \bz}$ with geometric convergence.

 Theorem~\ref{tm1} supports our interpretation that $p(x)^{-1/2}$ is a
 measure for the local \bw . If $p$ is constant on an interval $I$,
 $p|_I = p_0$, then the maximum gap condition~\eqref{eq:c2} reads as
 $x_{i+1}-x_i \leq \delta \sqrt{p_0} < \pi (p_0 / \Omega )^{1/2}$ 
 for $x_i \in I$. This is precisely the
 sufficient condition on the maximal gap  that arises for \bdl\
 functions with \bw\ $(\Omega / p_0)^{1/2}$. In other words, $f\in PW
 _{[0,\Omega]}(A_p)$ behaves like a  $(\Omega / p_0)^{1/2}$-\bdl\ function on $I$.

We remark that condition~\eqref{eq:c2} is almost optimal;  the constant
$\frac{\pi }{{\Omega ^{1/2}}}$ in \eqref{eq:c2} cannot be improved. However, a
weaker, qualitative version of this sampling theorem with a
sufficiently small $\delta $ in \eqref{eq:c2} can be derived from
Pesenson's theory of abstract  \bw~\cite{Pes01,Pesenson09}. 



Our second main result is a necessary density condition for sampling
in the style of Landau~\cite{Landau67, Landau67a}.  
For the formulation we need an adaptation of
the Beurling density to variable \bw . As in \eqref{eq:c2} we impose a
new measure or distance on $\br $ determined by the \bw\
parametrization $p$, namely
$ \mu_p(I)=\int_I p^{-1/2}(u) \, du$ 
and define the Beurling density of a set $X \subseteq \br$ as
\begin{equation*}
   D_p^-(X)= \varliminf_{r \to \infty}  \inf_{\mu_p (I) =r} \frac{\set{\# (X \cap I) \colon  I \subset \br \text{ closed interval } } }{r} \,.
\end{equation*} 
We write $\Lambda ^{1/2} = \{ \omega \in \br ^+ : \omega ^2 \in
\Lambda \}$ for the square root of a set and $|\Lambda ^{1/2}|$ for
its Lebesgue measure. 

\begin{thm} \label{tm2}
  Assume that  $p\in C^2$ and $p$ is eventually constant, i.e.,
  for some $R>0$ we have $p(x) = p_-$ for $x\leq -R$ and $p(x) = p_+$
  for $x\geq R$.  Fix $\Lambda  \subseteq \br
  ^+$ with finite (Lebesgue) measure. 
If $X\subseteq \br $ is a separated set such that the sampling
inequality
$$A \|f \|_2^2 \leq \sum _{i\in I} |f(x_i)|^2 \leq B \|f\|_2^2 
$$
holds  for all  $f \in PW_\Lambda  (A_p)$, then 
$D_p^- (X) \geq \frac{|\Lambda^{1/2}|}{\pi}$. 
\end{thm}

Theorem~\ref{tm2} is again consistent with our interpretation of
$PW_\Lambda  (A_p)$ as a space of functions with variable \bw . If
$\Lambda = [0,\Omega ]$ and $p$
is constant on an interval $I$, 
$p|_I = p_0$, then $\mu _p(I) = |I|/\sqrt{p_0}$ and we obtain roughly 
$$
\#
(X\cap I) \geq \frac{|\Lambda  ^{1/2}| |I|}{\pi \sqrt{p_0}} =
\Big(\frac{\Omega }{p_0} \Big)^{1/2} \frac{|I|}{\pi} \, .
$$
Comparing with Landau's classical result for \bdl\ function, this  is 
exactly the minimum number of samples in $I$ required for a \bdl\ function
with \bw\  $(\Omega /p_0)^{1/2}$. Again,  $f\in PW
 _{[0,\Omega]}(A_p)$ behaves like a  $(\Omega / p_0)^{1/2} $-\bdl\
 function on $I$. 


 The two main theorems  demonstrate convincingly  that
the spectral subspaces $PW_\Lambda (A_p)$ are indeed appropriate  models for
spaces of functions with variable \bw .  The values $p(x)^{-1/2}$ 
 may be taken  as a  measure for the local \bw\ and
enter significantly in the formulation of sampling and density
theorems for these spaces. 

The inherent analycity properties of functions of variable \bw\  (item
(iii) on our wishlist) follow from the theory of abstract
\bw\ \cite{Pes01,Pesenson09} and will be discussed in Section~3. 

\vspace{3mm}

\textbf{Methods.} The formulation of the main theorems looks like a small
variation of the standard theorems for classical \bdl\ functions 
 with the parametrizing function $p$ appearing in the appropriate places.  The
proofs  of the above theorems, however,   require input from
two areas, namely the  applied harmonic analysis of sampling  theory
and the detailed spectral analysis of Sturm-Liouville operators and
Schr\"odinger operators. The methodical input from  sampling theory is
the 
oscillation method from~\cite{Groch92} for the proof of
Theorem~\ref{tm1}, whereas  the proof of
Theorem~\ref{tm2} follows the outline of Nitzan and Olevski~\cite{Nitzan12} 
in which a (discrete) frame of reproducing kernels is compared to a
continuous resolution of the identity. 
The second methodical input  is from  the theory of Sturm-Liouville problems and of
(one-dimensional) Schr\"odinger operators. 
To see why we need the extensive build-up of
Sturm-Liouville theory, we recall that  much of the theory
of classical \bdl\ functions is based on the \ft\ and the explicit
formula for the reproducing kernel $k(x,y) = \tfrac{\sin (x-y)}{x-y}$ of the
standard Paley-Wiener space. Our
main effort is devoted to finding appropriate substitutes for
these explicit expressions. On the one hand, we find these in the
spectral theory of Sturm-Liouville operators. The detailed analysis of
the spectral measure of $A_p$ yields a representation of functions in
$PW_\Lambda (A_p)$ as
$$
f(x) = \int _\Lambda  F(\lambda ) \cdot \Phi (\lambda ,x) \, d\rho (\lambda
) \, ,
$$
where $\Phi (\lambda ,x) = \big( \Phi _1(\lambda , x), \Phi _2(\lambda
, x)\big)^T $ is a set of fundamental solutions of $-(p \Phi ' )'  = \lambda
\Phi $, $\rho$ is the $2\times 2$-matrix-valued spectral measure, and
$F \in L^2(\Lambda ,  d\rho)$. Though not as explicit as the \ft , this
spectral representation of functions of variable \bw\ will enable us
to derive the essential properties of $PW_\Lambda (A_p)$. 

For the proof of the  density theorem we will switch to an  equivalent
Schr\"odinger equation. By applying a Liouville transform, the
differential operator $-Dp(x) D$ is unitarily equivalent to a
one-dimensional Schr\"odinger operator $B_p = -D^2 + q(x)$ where $q$
is an explicit expression depending on the \bw\ parametrization $p$
(see~\eqref{eq:63} for the precise formula). The advantage of the
Schr\"odinger picture is that we can apply the scattering theory of
the Schr\"odinger equation to obtain asymptotic estimates for the
reproducing kernel of $PW_\Lambda (B_q)$. To appreciate the
transition to the Schr\"odinger picture, the reader should  at least
check  the  remark after the proof of Lemma~\ref{lem:conthap}. ~\ref{rem:alert}.

The additional  insight emerging
from this approach is that sampling sets for $PW_\Lambda (A_p)$ are
obtained from sampling sets for the  Paley-Wiener-space of the Schr\"o\-dinger operator $PW_\Lambda (B_q)$ (which is defined verbatim as in Definition~\ref{def1}) by means of
time-warping with the Liouville transform. Note, however, that the
concrete interpretation of variable \bw\ is lost in  the
Schr\"odinger picture. 

Let us emphasize that at this time we want to focus on the
proof of concept and to convince the reader that Definition~\ref{def1}
yields a meaningful and mathematically interesting notion of variable
\bw .  Therefore we develop the theory of variable \bw\ under
rather benign assumptions on the \bw\ parametrization $p$. It is clear
that on a technical level our results can and should be pushed much further  by
using more advanced aspects of 
Sturm-Liouville theory. Indeed, once the basic theory is established,
we face a host of new, interesting, and non-trivial  questions in
sampling theory. See the last section for a look at ongoing work. 

As the paper should be accessible for two different communities
(applied harmonic analysis and spectral analysis),  we  have tried to
work on a moderate technical level. 
This means in particular that we feel the
need to summarize some  well-known parts of the spectral theory to harmonic
analysts. 

\vspace{3mm}

\textbf{Related work and other notions of variable bandwidth.}  In the
literature one finds several approaches to variable \bw .

1. \emph{Time-warping} ~\cite{Clark85,Horiuchi68,Porat98,Shavelis12,Wei07}:
Given a homeomorphism $\gamma :\br \to \br$ (a warping function), a function $f$  possesses
variable bandwidth with respect to $\gamma$, if $f = g \circ \gamma$
for a bandlimited function $g\in L^2(\br )$ with $\supp \hat{g} \subseteq [-\Omega
, \Omega]$. The derivative $1/\gamma '(\gamma ^{-1}(x))$ of the
warping function is interpreted as the local bandwidth of $f$.
Although the sampling theory for time-warped functions is simple,
time-warping is relevant in signal processing, and the  estimate of  suitable warping functions for given
data is an important problem. Let us mention that time-warping
functions can be understood as spectral subspaces of certain
differential operators of order one (see Appendix~B).

2. \emph{Aceska and Feichtinger}~\cite{Aceska11,Aceska12} have proposed a concept  of variable
\bw\  based on  \tf\ methods, namely the truncation of the \stft\ by
means of a time-varying frequency cut-off. The resulting function
spaces, however,  coincide with  the standard Bessel-Sobolev potential
spaces endowed with an equivalent norm. Since  these spaces  do not admit a
 sampling theorem (nor  a Nyquist density), they are not
 variable \bw\ spaces in our sense.  The 
spaces of ~\cite{Aceska11,Aceska12} are rather spaces of locally variable smoothness.

3. \emph{Kempf and his collaborators}~\cite{Kempf00,Kempf08,Kempf10} use a
procedural concept of variable \bw , but (at least in the available
literature) shy away from a formal definition. 
The parametrization of self-adjoint extensions of a differential
operator leads to a class of algorithms that reconstruct or
interpolate a function  from certain
samples. The reconstructed function is said to have variable \bw .

4. \emph{Abstract Paley-Wiener spaces}~\cite{Pes01,Pesenson09,Zayed08}:
Perhaps closest to our approach is the  work
of Pesenson and Zayed on abstract bandlimitedness. Given an unbounded,  self-adjoint
operator on a Hilbert space $H$, the spectral subspaces $c_\Lambda
(A)H$ are considered abstract spaces of \bdl\ vectors. If $A$ is the
Laplace-Beltrami operator on a manifold, then  the corresponding
Paley-Wiener spaces are concrete function spaces and admit
sampling theorems.  These results are merely qualitative and so far  are not
backed up by corresponding density results.   Paley-Wiener spaces on
manifolds play an important role in the construction and analysis of
Besov spaces on various
manifolds. See~\cite{FilM11,Pes09,petrushev14,petrushev15} for
this direction of research. 


5. \emph{Sampling theory associated with Sturm-Liouville
  problems}~\cite{Zayed91}:
 The  generalizations of Kramer's sampling
theorem with Sturm-Liouville theory aim at interpolation formulas and
sampling theorems analogous to the cardinal series. However, the
samples are taken on the spectral side  and the conditions imposed on
the Sturm-Liouville operator guarantee a discrete spectrum, in
contrast to our set-up. 
Except for the use of Sturm-Liouville theory, we do not see any  connection
to our work.

\vspace{ 3mm}

\textbf{Organization.} The paper is organized as follows:  
In Section~\ref{specth}  we review the relevant aspects of the spectral theory of
Sturm-Liouville operators. In  Section~3 we collect  the 
basic properties of  functions of variable \bw .  
 Section \ref{sec:model-case} treats the toy
example of the discontinuous  parametrizing function $p(x) = p_-$ for $x\leq 0$
and $p(x) = p_+$ for $x>0$ and $p_- \neq p_+$. We will show that the  corresponding
Paley-Wiener space consists of functions with different  bandwidths
$1/\sqrt{p_{\pm}}$ on the positive and negative half axis and then   prove a
Shannon-like sampling theorem (Thm.~\ref{thm:shanlike}).  This example is
instructive because all objects (spectral measure, reproducing kernel)
can be computed  explicitly.  
In Section \ref{sec:irregular-sampling-} we prove Theorem~\ref{tm1}
and  treat the stable reconstruction of a function of variable \bw\
from a sampling sequence. The approach is based on  \cite{Groch92,
  grirrsstft}. 
In Section~\ref{sec:landau} we define a Beurling density adapted to
the sampling geometry of the differential operator and  prove
 necessary density conditions for stable sampling and
interpolation in the style of Landau. The proof follows the outline of
~\cite{Nitzan12}, the main technical work is to control the
reproducing kernels and its oscillations, which is done in
Section~7. 

The appendices contain our remarks on time-warping and  the explicit calculations needed for
Section~4. Such material is usually left to the interested reader, but
we prefer to include it, because we have struggled several times to
reproduce our own calculations. 


\vspace{ 3mm}

\textbf{Acknowledgment.} We would like to express our special  thanks
to Gerald Teschl for
many discussions and sharing his knowledge on Sturm-Liouville and
spectral theory.

\section{Spectral Theory of some second order differential operators}
\label{specth}
In this section we provide the necessary  background from the spectral
theory of second order differential operators, in particular, of 
Sturm-Liouville operators. Our standard references for the material
are \cite{Teschl09,Weidmann87,Weidmann03} and  also \cite{DS2}. 

We consider differential expressions of Sturm-Liouville type
\begin{equation}
  \label{eq:34}
  \tau f = - (pf')'+ q f
\end{equation}
where  $p,q$ are measurable functions satisfying 
\begin{equation}\label{eq:59}
 p > 0 \text{ a.e., and } 1/p, q \in L^1_{loc}(\br) \, .
 \end{equation}
 The corresponding \emph{maximal operator} $A$ is  defined by the
 choice of the domain 
\begin{equation}
\label{eq:domain}
  \begin{split}
    \mD(A) &=\set{f \in L^2(\br) \colon f, pf' \in AC_{loc},\tau f \in L^2(\br)} \\
    A f &= \tau f \quad \text{ for } \,  f \in \mD(A_p) \,.
  \end{split}
\end{equation}
{We will always impose the conditions \eqref{eq:59} and \eqref{eq:domain}  without further notice.}\\
For the study of variable \bw\ we will  focus  on differential expressions in divergence form,
\begin{equation}
  \label{eq:35}
  \tau_p f =-(pf')' \; 
\end{equation}
with  the corresponding maximal operator $A_p$. For the density
theorem  we will also  consider  Schr\"odinger operators
\begin{equation}
  \label{eq:36}
  \tilde\tau_q f =-f''+ q f \,
\end{equation}
with corresponding maximal operator $B_q$. Under mild  assumptions
$A_p$ and $B_q$ are unitarily equivalent via the Liouville transform, see Section~\ref{sec:transf-schr-form-1}.
\\ \\
To cite the necessary results from the  literature we need  some more  terminology and
notation. 

A solution $\phi$ of $(\tau-z)\phi=0$, $ z\in \bc$,
 \emph{lies left}  in $L^2(\br)$, if $\phi \in
 L^2((-\infty,c))$ for some  $c \in \br$, and lies right if  $\phi \in L^2((c,\infty))$
 for some  $c \in \br$. If for every $z \notin \br$ there are two
 unique (up to a multiplicative constant) solutions  of
 $(\tau-z)\phi=0$ that lie left respectively right in $L^2(\br)$
 (jargon: $\tau$ is  \emph{limit point} (LP) at $\pm \infty$) , then
 the corresponding maximal operator $A$ is \saj\ \cite[13.18,
 13.19]{Weidmann03}. This is the only situation we treat in this
 text. 
\\

We cite  a simple sufficient condition on $p$ such that the maximal operator $A_p$ corresponding to $\tau_p$ is \saj.

\begin{prop} \label{prop:saj}
Let $P(x) = \int_0^x \inv{ p(u)} du$. If $P\notin L^2(\br^+)$ and
$P\notin L^2(\br^-)$, then  $A_p$ is \saj .
\end{prop}
For a proof see~\cite[13.24]{Weidmann03}(or
    \cite[Thm.~6.3]{Weidmann87}) in conjunction with \cite[13.8, 13.19]{Weidmann03}.


\begin{rem}
 Under minimal additional assumptions, one can deduce more information
 about the spectrum of $A_p$. For instance, if 
 there exist  constants $C_1,C_2>0$ such that
\begin{equation*}
    \varliminf_{T\to \infty} \frac{1}{T} \int_0^T \bigabs{1-
      \frac{C_1}{p(u)}} du =0  \quad \text{ and } 
    \varliminf_{T\to \infty} \frac{1}{T}
    \int_{-T}^0 \bigabs{1- \frac{C_2}{p(u)}} du =0 \, , 
\end{equation*}
then $\sigma(A_p) =\sigma_{ess}(A_p)=[0,\infty )$, cf. ~\cite[Thm. 15.1]{Weidmann87}. 
 A sufficient condition for
$\sigma_{ac}(A_p) =\sigma_{ess}(A_p)=[0,\infty)$,
$\sigma_{sc}=\emptyset$ is given in~\cite[Thm 3.2]{Teschl08}.   
\end{rem}

In particular the conditions of Proposition~\ref{prop:saj} are satisfied for eventually constant
functions $p$, i.e., $p$ satisfies 
\begin{equation}
  \tag{MC}\label{eq:39}
  p(x) = 
\begin{cases}
p_-\,,\quad &x< -R\, \\
p_+\,, \quad &x> R\,
  \end{cases}
\end{equation}
for an $R>0$. We call this case  the \emph{model case}. 

For a \saj\ realization  $A$ of a  differential expression $\tau$ of
Sturm-Liouville type the functional  calculus can be described more
explicitly than by the spectral theorem alone. Our reference for most
of the following is mainly \cite{Weidmann03}, and  also \cite{Teschl09,DS2}.


Let  $\rho$ be a  positive semi-definite $ 2 \times 2$ matrix-valued
Borel measure (a \emph{positive matrix measure}),  
 and $L^2(\br, d\rho)$ the completion of the space of simple
 $\bc^2$-valued functions $F,G$ with respect to the scalar product  
 \begin{equation}
 \int_\br F(\lambda) \cdot \overline {G(\lambda)} d\rho(\lambda) 
=\int_{\br}\sum_{j,k=1}^2 F_{j}(\lambda) \overline {G_{k}\lambda)} d\rho_{jk}(\lambda).
 \end{equation}
Note that the trace $\mu=\operatorname{Tr} \rho$ is a positive Borel
measure and the components $\rho_{jk}$ are absolutely \cont\ with
respect to $\mu$. See \cite[XIII.5]{DS2} for basic properties of
matrix measures. 

\begin{thm}[{\cite[14.1, 14.3]{Weidmann03},  \cite[Lem.~9.28]{Teschl09}, \cite[XIII.5]{DS2}}]
  Assume that $\tau$ is a differential expression of \SL\ type in LP
  condition at $\pm \infty$ and that  $A$  is the corresponding \saj\ operator.  

If $\Phi(\lambda,x)=\big(\phi_1(\lambda,x), \phi_2(\lambda,x) \big)$ is a fundamental system of solutions of $(\tau-\lambda)\phi =0$
that depends \cont ly on $\lambda$, then there exists a $ 2 \times 2$ matrix  measure $\rho$, such that the  operator
\begin{equation}
  \label{eq:c21}
\mF_A: L^2(\br) \to L^2(\br, d\rho);
\quad \mF_Af(\lambda)=\int_\br f(x)\,\overline{ \Phi(\lambda,x)} dx
\end{equation}
is unitary and diagonalizes $A$, i.e., 
$
\mF_A A \inv{\mF_A} G (\lambda)= \lambda G(\lambda) 
$
for all $G \in L^2(\br, d\rho)$. 
 The inverse has the form
\[
\inv{\mF_A} G (x) =\int_{\br} G (\lambda)\cdot {\Phi(\lambda,x)} \,d\rho(\lambda)\,.
\]
for $G \in L^2(\br, d\rho)$.

 If  $g$ is a bounded Borel function   on $\br$ then, for every $f \in \Ltwo$,
\begin{equation}
  \label{eq:specrep}
  g(A)f (x) = \int_{\brp} g(\lambda) \mF_A f(\lambda) \cdot {\Phi(\lambda,x)} \, d \rho(\lambda)\,.
\end{equation}
All integrals $\int_\br$ have to be understood as $\lim_{\substack{a
    \to -\infty \\ b \to \infty}}\int_{a}^b$ with convergence in
$L^2$. 
\end{thm}
Note that  the spectral projection $c_\Lambda (A)$ is given by $
c_\Lambda (A)f (x) = \int_{\Lambda}  \mF_A f(\lambda) \cdot
{\Phi(\lambda,x)} \, d \rho(\lambda)$. 

$\mF_A$ is called the \emph{spectral transform} (or also  a spectral representation of $A$).
  \begin{rems}
1.  It is always possible to choose a fundamental system of
    solutions $\Phi (\lambda   ,\cdot )$ that depends continuously (even
    analytically)  on
    $\lambda $~\cite{Weidmann03}.
  The spectral measure can then  be  constructed explicitly from the
    knowledge of such a  set of  fundamental solutions $(A-z)\phi =
    0$, see \cite{DS2,Teschl09,Weidmann87,Weidmann03}. We will explain
    some details of this construction in  Appendix~A. 

2.   It can be shown that under mild 
     conditions the measure $d\rho $  is absolutely
     \cont\ with respect to Lebesgue measure, see
     ~\cite[Thm. 3.2]{Teschl08}. 
  \end{rems}

\section{Basic properties of Paley-Wiener functions}
\label{sec:basics}
In this section we define  the Paley-Wiener spaces and describe  some of their elementary properties.
\begin{defn} \label{pesbw}
Assume that  $\Lambda \subseteq \br$ with finite measure and that  $A$
is a \saj\ realization of a differental expression of \SL\ type. A
function $f \in L^2(\br)$ is in the Paley-Wiener space $\PW [A]
{\Lambda} $ (or $\Lambda$-\BL\ with respect to $A$), if  
  \begin{equation}\label{eq:abl}
    f= 
    c_{\Lambda}(A) f \, . 
  \end{equation}
\end{defn}

Definition~\ref{pesbw} is a special case of Pesenson's  abstract
notion of \bw\ ~\cite{Pes01} for  general \saj\ operators on a Hilbert
space. Subsequently Paley-Wiener spaces became an important notion in many
investigations in analysis,
see~\cite{grkl10,Pesenson09,Zayed08,petrushev15} for some examples. 
Our main contribution is the detailed investigation of the Paley-Wiener space associated to the
Sturm-Liouville operator $A_p = - (pf')'$ and their interpretation  as spaces of variable
\bw . Our  main interest is  the subtle  dependence of these spaces on
the  \bw\ parametrization $p$ and corresponding sampling results.

Using the spectral theory of Sturm-Liouville operators, the
Paley-Wiener spaces $PW_\Lambda (A)$ possess  characterizations
similar to  the standard spaces of \bdl\ functions.

\begin{prop} \label{prop:charbw}
  Assume that $\Lambda \subseteq \br ^+_0$ with finite measure and that
  $A \geq 0$ is  a \saj\
realization of a differential expression of \SL\ type with spectral
measure $\rho$ and corresponding spectral transform $\mF _A$. Then the
following are equivalent:

(i) $f\in PW_\Lambda (A)$,

(ii) $\mathrm{supp}\, \mF _A f \subseteq \Lambda$, 

(iii) there exists  a function $F \in L^2(\Lambda, d\rho)$ such that
\begin{equation} 
f(x)=  \int_\Lambda F(\lambda) \cdot \Phi(\lambda,x) \,
d\rho(\lambda) \qquad \text{ a.e. } \, x\in \br \, .\label{eq:6}
 \end{equation}

If the spectral set is an interval,  $\Lambda = [0,\Omega ]$, then
also the following conditions are equivalent to (i) --- (iii):

(iv) \label{item:6} Bernstein's inequality:   $\norm{A^k f}
\leq  \Omega^k \norm{f}$ for all $k \in \bn$.

(v) Analyticity: For all $g\in L^2(\br )$ the function $z\in \bc \to
\langle e^{zA} f, g\rangle $ is an entire function of exponential type
$\Omega $, i.e., for all $\epsilon >0$
$$
|\langle e^{zA}f, g\rangle | = \mO (e^{(\Omega +\epsilon ) |\Im z|}) \, .
$$
\end{prop}

\begin{proof}
(i) $\Rightarrow$ (ii), (iii): Since  $f
= c_\Lambda (A)f$ for  $f\in PW_\Lambda (A)$, \eqref{eq:specrep}
implies that 
$$
f(x) =  \int_{\brp}  \mF_A f(\lambda) \cdot
{\Phi(\lambda,x)} \, d \rho(\lambda) =  c _\Lambda (A)f (x) =
\int_{\brp} c _\Lambda (\lambda) \mF_A f(\lambda) \cdot {\Phi(\lambda,x)} \, d \rho(\lambda)\,.
$$
Since $\mF _A ^{-1} $ is unitary and thus one-to-one, this identity implies
that $\mF_A f(\lambda ) = c _\Lambda (\lambda ) \mF_A f(\lambda ) $
for $\rho$-almost all $\lambda $, whence $\supp \, \mF _A f \subseteq
\Lambda $. 

     (iii) $\Rightarrow$ (i) Conversely, if $f$ is  represented by
    \eqref{eq:6}, then $c _\Lambda (A) f = f$ and thus $f\in
    PW_\Lambda (A)$. 

The equivalence (i) $\Leftrightarrow$ (iv) 
follows directly from the spectral theorem for self-adjoint unbounded
operators and was first proved in~\cite{Pes01}    for an  abstract notion
of bandlimitedness. The characterization
(i) $\Leftrightarrow$ (v) is proved in \cite{Pesenson09, Zayed08}.
See also~\cite{grkl10} for a related characterization of \bw .
\end{proof}

\begin{prop} \label{prop:basicpw} 
Let $\Lambda$ be a subset of $\br ^+ _0$ with finite measure, and $A \geq 0$ a \saj\
realization of a differential expression of \SL\ type. 

(i) Then   \label{item:4} the Paley-Wiener space $\PW [A] \Lambda$ is
a closed subspace of $\Ltwo$. Every function in $\PW [A] \Lambda $ is
continuous. 

(ii) \label{item:8} If $\Lambda$ is compact, the Paley-Wiener space $\PW [A]\Lambda$ is a
reproducing kernel Hilbert space; its  kernel  is 
 \begin{equation} 
k_\Lambda (x,y)=k (x,y)=\int_\Lambda \overline{ \Phi(\lambda,x)} \cdot
{\Phi(\lambda,y)}\, d \rho(\lambda) \, ,\label{eq:45} 
\end{equation}
and $k$ is the integral kernel of the spectral projection $c_\Lambda(A)$ from
$L^2(\br )$ onto $\PW [A]\Lambda $.  The kernel $k$ is \cont\ in $x$ and $y$.
\end{prop}
\begin{proof}
We apply the Cauchy-Schwarz inequality for $L^2(\br , d\rho
)$~\cite[XIII.5.8]{DS2} to \eqref{eq:6} and obtain 
\begin{equation} \label{eq:7}
\abs {f(x)} \leq \norm{F}_{L^2(\Lambda, d\rho)}
\norm{\Phi(\cdot,x)}_{L^2(\Lambda, d\rho)} \, .
\end{equation}
Thus the pointwise evaluation $f \mapsto f(x) $ is continuous on
$PW_\Lambda (A)$ and  $PW_\Lambda (A)$ is a reproducing kernel Hilbert
space. Since  $\Phi(\lambda,x)$ is \cont\ in $\lambda$, the continuity
of $f$ 
follows from classical facts about parameter integrals. 

The formula \eqref{eq:45} for the reproducing kernel is proved
in~\cite[XIII.5.24]{DS2}.  It follows directly from the identity
\begin{equation*}\label{eq:bl}
  \begin{split}
    P_\Lambda(A)f (x)&= \int_\Lambda \mF_A f(\lambda)\cdot\overline{\Phi(\lambda,x)} d \rho(\lambda) \\
   &=\int_\Lambda \int_\br f(y) { \Phi(\lambda,y)}\cdot \overline{\Phi(\lambda,x)} d \rho(\lambda) dy\\
   &= \int_\br f(y) \,\overline{k_\Omega (x,y)}\,dy  \,,
    \end{split}
 \end{equation*}
after justifying the interchange of the integrals. 
\end{proof}

\begin{rem}
  The compactness condition in \eqref{item:8} above can be relaxed in
  various important cases.  The inequality ~\eqref{eq:7} holds if
  $\norm{\Phi(\cdot,x)}_{L^2(\Lambda, d\rho)}$ is finite. This is the
  case under the following set of conditions: (i) $\abs \Lambda <
  \infty$, (ii)  the spectral measure is
  absolutely \cont\ with respect to Lebesgue measure, and (iii)  the
  solutions $\Phi(\cdot,x)$ are bounded for every $x$. In particular
  this is true for the Schr\"odinger operator \eqref{eq:36} with
  compactly supported potential $q$. We will use this  fact in 
  Section~\ref{sec:landau} (Proposition \ref{Schroescatt} and
  Equation \eqref{eq:specfun}). 
\end{rem}

As mentioned in the introduction, a function in $PW _{[0,\Omega
  ]}(A_p)$ behaves locally like a function with \bw\ $(\Omega
/p(x))^{1/2}$. The next result provides a precise formulation for this vague
idea. Before its statement we recall that a function $f$ belongs to
the Bernstein space $B_\Omega$, if its distributional
Fourier transforms has support in $[-\Omega, \Omega]$. By the
Paley-Wiener theorem for distributions (see, e.g., \cite{Rudin73})
$f\in B_\Omega $, \fif\  $f$ can be extended from $\br$ to an entire
function $F$ of exponential type $\Omega $, i.e., $F|_\br = f$ and
$|F(z)| = \mO (e^{(\Omega +\epsilon ) |\mathrm{Im}\, z|})$. 

\begin{prop} \label{prop:effbw} 
  If $p(x)=p_0$ for $x$ in an open interval $I$, then on $I$  every $f \in
  PW_{[0,\Omega]}(A_p)$  coincides with a function in
  $B_{\sqrt{\Omega/p_0}}$ restricted to $I$.  
\end{prop}
\begin{proof}
Since $p|_I = p_0$, the restriction of the differential expression
$\tau _p$ to $I$ is just $-p_0 \tfrac{d^2}{dx^2}$ and therefore the
eigenvalue equation $( \tau _p - \lambda )\phi = 0$ possesses the
solutions $e^{\pm i \sqrt{\lambda/p_0} x}$ on $I$. Let $\tilde \Phi
  (\lambda , x) = (e^{i \sqrt{\lambda/p_0} x }, e^{- i
      \sqrt{\lambda/p_0} x})$ for \emph{all} $x\in \br $, i.e.,
    $\tilde \Phi
    (\lambda , \cdot )$ is a
      fundamental system of $(-p_0 \tfrac{d^2}{dx^2} - \lambda )\phi =
      0$ on $\br $. We may now  choose  a fundamental system $\Phi
      (\lambda ,  \cdot )$ of $(\tau _p - \lambda ) \phi =0$ and a
      corresponding spectral measure $\rho $, such
      that $\Phi (\lambda , x) = \tilde \Phi (\lambda ,x)$ for $x\in
      I$. By ~\eqref{eq:6} every $f\in PW_\Lambda (A_p)$ has a
      representation $f(x) =   \int _0 ^\Omega  F(\lambda) \cdot \Phi(\lambda,x) \, d\rho(\lambda)$
for some $F \in L^2([0,\Omega ]  ,d\rho )$.  Now set 
$$
\tilde f(x) =   \int _0 ^\Omega  F(\lambda) \cdot \tilde
\Phi(\lambda,x) \, d\rho(\lambda) \, . 
$$
Then $f(x) = \tilde f(x) $ for $x\in I$. 

By definition, $\tilde \Phi (\lambda , \cdot )$ can be extended to an
entire function that satisfies $|\tilde \Phi (\lambda , z)| \leq
2 e^{\lambda |\Im z|}$. Therefore the vector-valued H\"older inequality
implies that $\tilde f$ also extends to a function on $\bc $ obeying the
growth estimate 
\begin{equation}\label{eq:4}
\abs{\tilde f(z)} \leq \norm{F}_{L^1([0,\Omega], d\rho)} \norm{\Phi(\cdot,
  z)}_{L^\infty([0,\Omega], d\rho)}\leq   C e^{\sqrt{ \frac \Omega
    {p_0}  }\abs{\Im z}} \, 
\end{equation}
for  $z \in \bc$.
Clearly, the restriction of $\tilde f$  to $\br $ is bounded, and the
parameter integral $z\mapsto \tilde f(z)$ is analytic in 
$z$ (use Morera's theorem). Therefore $\tilde f$ belongs to the
Bernstein space $B_{\sqrt{\Omega / p_0}}$ and $\tilde f| _I = f |_I$,
as claimed. 
\end{proof}

\section{A Toy Example}
\label{sec:model-case}
In this section we assume that    $\Lambda=[0,\Omega] \subseteq \br^+$ and
 study the  special case 
\[
p(x)= \begin{cases}
p_-, &\quad x\leq 0\\
p_+, &\quad x >0 \,.
\end{cases}
\]
Our point is that all formulas are explicit and thus may help build
the reader's intuition for variable \bw . 

Using the continuity of solutions $\phi$ of
the differential equation
$(\tau_p-z)\phi =0$
and the continuity of $p(x)\phi'(x)$ at $x=0$, 
we obtain the linearly independent solutions
\[
\phi_+(z,x)=\begin{cases}
\frac{1}{2} \Big(1+ \sqrt{\frac{p_+}{p_-}}\Big) e^{i  \sqrt{ z /{p_-}} x}+ \frac{1}{2} \Big(1- \sqrt{\frac{p_+}{p_-}}\Big) e^{-i \sqrt{z/p_-} x}, \quad & x \leq 0,\\
e^{i  \sqrt{z/p_+} x}, \quad & x>0 \;,\\
\end{cases}
\]
\[
\phi_-(z,x)=\begin{cases}
e^{-i  \sqrt{z/p_-} x}, \quad & x \leq 0,\\
\frac{1}{2} \Big(1+ \sqrt{\frac{p_-}{p_+}}\Big) e^{-i  \sqrt{z/p_+} x}+  \frac{1}{2} \Big(1- \sqrt{\frac{p_-}{p_+}}\Big)e^{i  \sqrt{z/p_+} x}, \quad & x>0 \;,\\
\end{cases}
\]
by a straightforward calculation. 
For $\Im z >0$ the solution $\phi_+$ lies right in $L^2(\br)$ and $\phi_-$ lies left in $L^2(\br)$.
Note that for real $z=\lambda \in \br $ and $x\leq 0$  the solution
$\phi _+$ can be written as
\[
\phi _+ (\lambda , x) = \cos \big( \sqrt{\tfrac\lambda { p_-}} x\big) + i
\sqrt{\tfrac{p_+}{p_-}} \sin  \big( \sqrt{\tfrac \lambda { p_-}} x\big) \, ,
\]
and similarly for $\phi _-$ for $x\geq 0$.

We can
derive the spectral measure  explicitly to be 
\begin{equation}
  \label{eq:22}
 d \rho(\lambda)=\frac{1}{\pi (\sqrt{p_-}+\sqrt{p_+})^2}
\begin{pmatrix}
\sqrt{p_-} & 0 \\
 0 & \sqrt{p_+}
\end{pmatrix}
\, \frac{d\lambda}{\sqrt{\lambda}} \,. 
\end{equation}
 The details of this computation are sketched  in  Appendix~A.

Using~\eqref{eq:45} and $\sinc (x) = \sin x /
x$,   a straightforward computation leads to the
the reproducing kernel 
\begin{equation}\label{eq:2}
  k(x,y) = 
\begin{cases}
      \frac {\Omega^{1/2}} {\pi \pmi}\sinc(\Omega^{1/2} \frac
      {x-y}{\pmi})- \frac{\pp-\pmi}{\pp+\pmi} \frac {\Omega^{1/2}}{\pi
        \pmi} \sinc({\Omega^{1/2}} \frac    {x+y}{\pmi})  \qquad &x, y
      \leq 0 \,,   \\ 
 \frac {\Omega^{1/2}} {\pi \pp}\sinc({\Omega^{1/2}} \frac {x-y}{\pp})+
 \frac{\pp-\pmi}{\pp+\pmi} \frac {\Omega^{1/2}}{\pi \pp}
 \sinc({\Omega^{1/2}} \frac    {x+y}{\pp}) \qquad& x, y > 0 \,,\\ 
\frac{2 {\Omega^{1/2}}}{\pi (\pp +\pmi)} {\sinc\big({\Omega^{1/2}}
  (\frac x {\pmi}- \frac{y} {\pp})\big)} \qquad& x\leq 0, y>0 \,,\\ 
\frac{2 \Omega^{1/2}}{  \pi (\pp+ \pmi)} \sinc\big({\Omega^{1/2}}(\frac x {\pp}- \frac{y} {\pmi})\big)   \qquad & x >0, y\leq 0 \,.
\end{cases}
\end{equation}
If $p_-=p_+=1$ we obtain the $\sinc$ kernel, as expected.

Direct inspection of~\eqref{eq:2} suggests a recipe to obtain an orthonormal basis 
of $PW_{[0,\Omega ]}(A_p)$. As a result we obtain  a sampling formula that
is similar to the cardinal series for \bdl\ functions.
\begin{thm}\label{thm:shanlike}
Fix the sampling set $\set{x_j =\frac{ \pi j\sqrt{p(j)} }{\Omega
    ^{1/2}}: j \in \bz } $ and the weights 
\begin{equation}
  \label{samplingweights}
  w_j=
  \begin{cases}
    \pmi\,,&\quad j <0\,, \\
    \pp\,, &\quad j >0\,, \\
    \frac 1 2 (\pp+\pmi)\,, &\quad j=0\,.
  \end{cases}
\end{equation}
Then  the set  $\{\sqrt{\frac{\pi w_j}{\Omega^{1/2}}}k(x_j,\cdot) \colon j
\in \bz\}$  forms an orthonormal basis  of $PW_{[0,\Omega]}(A_p)$. The orthogonal expansion
  \begin{equation}
    f(x) = \frac{\pi} {\Omega^{1/2}} \sum_{j \in \bz}{w_j} f(x_j) {k(x_j,x )}
  \end{equation}
converges  in $L^2(\br)$ and uniformly for every $f \in PW_{[0,\Omega]}(A_p)$. 
\end{thm}
\begin{proof}
The orthogonality follows directly from  $\langle k(x_i,
\cdot ) , k(x_j,\cdot )\rangle = k(x_i,x_j)$ and the formulas for the
reproducing kernel~\eqref{eq:2}. For $i=j$ we obtain $\langle k(x_j,
\cdot ) , k(x_j,\cdot )\rangle = k(x_j,x_j)= \tfrac{\Omega ^{1/2}}{\pi
  w_j}=\tfrac{\Omega ^{1/2}}{\pi \sqrt{p_{\pm}} }$, and thus  
$\{\sqrt{\frac{\pi w_j}{\Omega^{1/2}}}k(x_j,\cdot) : j\in \bz \}$
is  an orthonormal set. 

To prove the completeness of the orthogonal system, assume that $f \in
PW_{\Lambda}(A_p)$ and $\langle f, k(x_j,\cdot  )\rangle = 0$ for all
$j\in \bz $. We have to show that $f \equiv 0$.

Using the unitarity of $\mF _{A_p}$, this is equivalent to proving that  
$F = (F_1, F_2)  \in L^2(\Lambda, d\rho) $  and 
$\inprod{ F, \Phi(\cdot,x_j,)}_{L^2(\Lambda, d\rho)} = 0 $
 for all $j \in \bz$  implies $F \equiv 0$.
 
We  substitute the fundamental solutions in the inner product and
make the  change of
variables $\lambda = \omega ^2, d\lambda / \sqrt{\lambda } = 2 d\omega
$; then the vanishing of  $\inprod{ F, \Phi(\cdot,x_j,)} = 0$ amounts  to
the following conditions: 
    \begin{equation*}\label{eq:23}
    \int_{0}^{\Omega^{1/2}}\big(\pmi F_1(\omega^2)+ \pp
    F_2(\omega^2)\big) \cos \frac{\pi j \omega}{{
        \Omega ^{1/2} }}\;d\omega  
+ i \pp \int_{0}^{\Omega^{1/2}}\big( F_1(\omega^2)- F_2(\omega^2) \big)  \sin \frac{\pi  j \omega}{{\Omega^{1/2}}}\;d\omega  =0, \quad j <0 \,,
\end{equation*}
\begin{equation*}\label{eq:24}
\int_{0}^{\Omega^{1/2}}\big(\pmi F_1(\omega^2)+ \pp F_2(\omega^2)\big) \cos \frac{\pi j \omega}{{\Omega^{1/2}}}\;d\omega 
+ i \pmi \int_{0}^{\Omega^{1/2}}\big( F_1(\omega^2)- F_2(\omega^2) \big)  \sin \frac{\pi j \omega}{{\Omega^{1/2}}}\;d\omega  =0, \quad j \geq 0 \,,
  \end{equation*}
If we re-index the first equation 
($j \to -j$) we obtain
\begin{equation*}
  \label{eq:25}
   \int_{0}^{\Omega^{1/2}}\big(\pmi F_1(\omega^2)+ \pp F_2(\omega^2)\big) \cos \frac{\pi j \omega}{{\Omega^{1/2}}}\;d\omega 
- i \pp \int_{0}^{\Omega^{1/2}}\big( F_1(\omega^2)- F_2(\omega^2) \big)  \sin \frac{\pi j \omega}{{\Omega^{1/2}}}\;d\omega  =0, \quad j >0 \;.
\end{equation*}
 Adding and subtracting the above equations yields
  \begin{align}
    \label{eq:26}
    &\int_{0}^{\Omega^{1/2}}\big(\pmi F_1(\omega^2)+ \pp
    F_2(\omega^2)\big) \cos \frac{\pi j \omega}{{
        \Omega ^{1/2}}}\;d\omega =0, \quad j\geq 0 \,, \\ 
    \label{eq:27}
  &  \int_{0}^{\Omega^{1/2}}\big( F_1(\omega^2)- F_2(\omega^2) \big)  \sin \frac{\pi j \omega}{{\Omega^{1/2}}}\;d\omega  =0, \quad j >0 \;.
  \end{align}
Equation \eqref{eq:26} describes  the Fourier cosine coefficients of
the even continuation of the function   $\omega \mapsto \pmi
F_1(\omega^2)+ \pp F_2(\omega^2)$  to $[-{\Omega^{1/2}}, {
  \Omega^{1/2}}]$. Consequently,  from  the uniqueness of Fourier cosine
series we obtain  $ \pmi
F_1(\omega^2)+ \pp F_2(\omega^2)  = 0$ on $[0, {
  \Omega^{1/2}}]$. Likewise, \eqref{eq:27} with an odd extension of the
integrand,   the uniqueness 
theorem for  Fourier sine series yields  
$    F_1(\omega^2)- F_2(\omega^2)  =0 $
on $[0,{\Omega^{1/2}}]$. Combining these two results and substituting
back we obtain $F_1=F_2=0$ on $[0,\Omega]$. 

We have proved that  the set   $\{\sqrt{\frac{\pi w_j}{{
      \Omega^{1/2}}}} k(x_j,\cdot) :  j 
\in \bz\}$ is an orthonormal basis  of $PW_{\Lambda}(A_p)$.  Therefore every  $f \in
PW_\Lambda(A_p)$ possesses the orthonormal  expansion
\begin{equation*}
f= \sum_{j\in \bz} \frac{\pi w_j}{{\Omega^{1/2}}}\inprod{f, k
    (x_j,\cdot)} k(x_j,\cdot)= {\frac{\pi}{{\Omega^{1/2}}}}\sum_{j\in
  \bz} w_j f(x_j)  k (x_j,\cdot)  
\end{equation*}
with convergence in norm. In a reproducing kernel Hilbert space, the
$L^2$-convergence  implies pointwise convergence. Since  the family
$\set{k(x_j,\cdot) \colon j \in \bz}$ is norm-bounded,  it also implies
uniform convergence  (see~\cite[3.1]{Nashed91}). 
\end{proof}
\begin{rem}
 It would be of great interest to construct an orthonormal basis of
 reproducing kernels  and  a corresponding  similar sampling theorem
 for more general control functions $p$. 
\end{rem}

Let us now consider  the case $p_- \to \infty$, $p_+=1$. 
Proceeding formally from \eqref{eq:2}, the reproducing kernel of $PW_{[0,\Omega]}(A_p)$ is
\[
k(x,y)=
\begin{cases}
\frac{{\Omega^{1/2}}}{\pi} [ \sinc({\Omega^{1/2}}  (x-y))- \sinc
({\Omega^{1/2}}  (x+y))]\,,& \quad x, y \geq 0, \\ 
0 \,,&\quad \text{else .}
\end{cases}
\]
We see that every $f\in PW_{[0,\Omega]}(A_p)$ has its support in
$[0,\infty )$. For a more concrete interpretation of
$PW_{[0,\Omega]}(A_p)$ let $g$ be the extension of $f$ to an odd
function on $\br $. Then $f\in PW_{[0,\Omega]}(A_p)$ holds, \fif\ 
$$
f(x) = \int _0^\infty f(y) k(x,y) \, dy = \int _{\br } g(y)
\sinc({\Omega^{1/2}}  (x-y)) \, dy \,  \quad x \geq 0 \, . 
$$
This means that $f$ can be interpreted as  the restriction of an odd \bdl\
function $g$ with $\supp \, \hat{g} \subseteq [-\Omega ^{1/2},
\Omega ^{1/2}]$ to $[0,\infty )$.

Theorem~\ref{thm:shanlike}  then yields  the following sampling theorem. 
\begin{prop}\label{cor:st}
 If $p(x)=\infty$ for $x \leq{0}$ and $p(x)=1$ for
 $x>0$,  the following sampling formula holds for $f \in
 PW_{[0,\Omega]}(A_p)$  
\begin{equation*} 
    f(x) = \frac{\pi} {\Omega^{1/2}} \sum_{j =1}^\infty f(\frac{\pi j}{{\Omega^{1/2}} }) {k(\frac{\pi j}{{\Omega^{1/2}} },x )} = \sum_{j=1}^\infty (-1)^jf(\frac{\pi j}{{\Omega^{1/2}} }) \sin({\Omega^{1/2}} x) \frac{2 \pi j}{({\Omega^{1/2}} x)^2- (\pi j)^2} \,.
  \end{equation*}
\end{prop}
We do not give a formal proof for the limiting procedure. The
corollary follows directly from the observation that $f \in PW_{[0,\Omega]}(A_p)$, \fif\ 
$f$ is the restriction of an odd function $g$  with $\supp \hat{g}
\subseteq [-\Omega ^{1/2}, \Omega ^{1/2} ]$ to $\br ^+$. Thus 
\begin{align*}
f(x) &=  \tfrac{\pi }{\Omega ^{1/2}}\,
\sum _{j\in \bz} f\big( \tfrac{\pi j}{\Omega ^{1/2}}   \big)
k\big(\tfrac{\pi j}{\Omega ^{1/2}}, x \big) \\
&= \tfrac{\pi }{\Omega ^{1/2}}\, \sum _{j = 1} ^\infty  f\big(
\tfrac{\pi j}{\Omega ^{1/2} }   \big) \Big( \mathrm{sinc}\, \big(
\Omega ^{1/2} (x -\tfrac{\pi j }{\Omega ^{1/2}}) - \mathrm{sinc}\,
\big( \Omega ^{1/2}
(x+\tfrac{\pi j }{\Omega ^{1/2}}) \big) \Big) \, .
  \end{align*}

%

\section{Nonuniform  Sampling }
\label{sec:irregular-sampling-}
In this section we assume that  $\Lambda \subseteq [0,\Omega] \subset \br^+$. Based on the method
developed  
in~\cite{Groch92,grirrsstft},  we derive a
sampling theorem  and reconstruction procedures for $\PW \Lambda$. 
All constants and error estimates are explicit and highlight the role
of the parametrizing function $p$. 

Given a set $X=\set{x_i \colon i \in \bz} \subseteq \br $, we   denote by
$
y_i= \frac 12 (x_i+ x_{i+1})  , i\in \bz \, , 
$
the midpoints of  $X$, and set $\chi_i=c_{[y_{i-1},y_i)}$.  
Then the functions $\chi_i$ form a partition of unity.            
 We also  set $I_i'=[y_{i-1},x_i)$ and $I_i''=[x_i, y_{i})$.
 Let $\delta $ be the maximum gap between 
consecutive sampling points weighted by the parametrizing function
$p$, namely
\begin{equation} \label{eq:c24}
\delta =\sup_{i\in \bz }  \frac{x_{i+1}-x_i} {\inf_{x \in
    [x_i,x_{i+1}]} \sqrt{p(x)}} \, .
\end{equation}

We first derive a fundamental inequality for functions  in  $PW_\Lambda (A_p)$. 
For the proof we need Wirtinger's inequality, see, e.g., \cite{inequalities}.
If $f,f' \in L^2([a,b])$ and either $f(a)=0$ or $f(b)=0$, then
\begin{equation}
  \label{eq:46}
  \int_a^b \abs{f(x)}^2 dx \leq \frac 4 {\pi^2} (b-a)^2 \int_a^b\abs{f'(x)}^2 dx \,.
\end{equation}

\begin{lem} \label{l:pp}
Let $\Lambda \subseteq [0,\Omega] \subset \br ^+$ and assume that
$\inf _{x\in  \br} p(x) >0 $. If $\delta $ is finite, then for all
$f\in PW _\Lambda (A_p)$  we have 
\begin{equation}
  \label{eq:c25}
    \bignorm{f-\sum_{i \in \bz} f(x_i) \chi_i }^2 \leq \frac{\delta ^2
    \Omega}{\pi ^2} \norm{f}^2 \, .
\end{equation}
Consequently, $PW_\Lambda (A_p)$ satisfies  a 
Plancherel-Polya inequality of the form
\begin{equation} \label{eq:c26}
  \sum_{i \in \bz} \abs{f(x_i)}^2\frac{x_{i+1}-x_{i-1}}{2} \leq \big( 1+ \frac{\delta 
    \Omega ^{1/2}}{\pi } \Big)^2  \norm{f}^2 \, , \qquad   f\in
  PW_\Lambda (A_p)   \, .
\end{equation}
\end{lem}

\begin{proof}
  The proof is an adaption of the proof of ~\cite[Thm.~1]{Groch92}. We
  rewrite the expression \eqref{eq:c25} as follows: 
 \begin{align} 
          \bignorm{f-\sum_{i \in \bz} f(x_i) \chi_i }^2 
       =& \bignorm{\sum_{i \in \bz}(f- f(x_i)) \chi_i }^2 = \sum_{i
         \in \bz}\int_{y_{i-1}}^{y_{i}}\abs{f(x)- f(x_i)}^2 dx \notag \\ 
       =& \sum_{i \in \bz}\Bigg(\int_{y_{i-1}}^{x_i}\abs{f(x)-
         f(x_i)}^2 dx + \int_{x_i}^{y_{i}}\abs{f(x)- f(x_i)}^2 dx
       \Bigg) \,.  \label{eq: Wirtinger}
     \end{align}
By Wirtinger's inequality ~\eqref{eq:46} this can be estimated further as
\begin{align*}%
 \bignorm{f-\sum_{i \in \bz} f(x_i) \chi_i }^2    \leq& \frac{4}{\pi^2}\sum_{i \in \bz}
 \Big[(x_i-y_{i-1})^2\int_{y_{i-1}}^{x_i}\abs{f'(x)}^2 dx + (y_i-x_i)^2\int_{x_i}^{y_{i}}\abs{f'(x)}^2 dx \Big]\\
\leq& \frac{4}{\pi^2}\sum_{i \in \bz}
 \Big[\frac{(x_i-y_{i-1})^2}{\min_{x \in I_i'} p(x)}\int_{y_{i-1}}^{x_i}p(x)\abs{f'(x)}^2 dx + \frac{(y_i-x_i)^2}{\min_{x \in I_i''} p(x)}\int_{x_i}^{y_{i}}p(x)\abs{f'(x)}^2 dx \Big]\\
\leq & \frac{4}{\pi^2}\sup_{i \in \bz} \max \Big(\frac{(x_i-y_{i-1})^2}{\min_{x \in I_i'} p(x)},\frac{(y_i-x_i)^2}{\min_{x \in I_i''} p(x)} \Big)
\sum_{i \in \bz} \int_{y_{i-1}}^{y_i}  p(x)\abs{f'(x)}^2 dx \\
= & \frac{4}{\pi^2}\sup_{i \in \bz} \Big(\frac{(x_i-y_{i-1})^2}{\min_{x \in I_i'} p(x)},\frac{(y_i-x_i)^2}{\min_{x \in I_i''} p(x)} \Big)
 \int_{\br}  p(x)\abs{f'(x)}^2 dx \\ 
\leq &\frac{1}{\pi^2}\sup_{i \in \bz}\max\frac{(y_i-y_{i-1})^2}{\min_{x \in [x_{i-1},x_i)} p(x)}  \int_{\br}  p(x)\abs{f'(x)}^2 dx \,.
\end{align*}
As $f \in \PW \Lambda \subseteq \mD (A_p)$ we can simplify the last term
by  using integration by parts and then apply Bernstein's inequality
(Proposition~\ref{prop:charbw} (iv)). 
\begin{equation}
  \label{eq:100}
  \int_{\br}  p(x)\abs{f'(x)}^2 dx=\inprod{A_p f, f} \leq
  \norm{A_pf}\norm{f} \leq \Omega \, \norm{f}^2 \, .
\end{equation}
The decisive modification was to smuggle in the parametrizing function
$p$ to obtain $\int p |f'|^2  $ and then to apply Bernstein's
inequality.

In conclusion  we obtain
\begin{equation}\label{eq:8}
 \bignorm{f-\sum_{i \in \bz} f(x_i) \chi_i }^2 \leq
 \frac{1}{\pi^2}\sup_{i \in \bz}\frac{(x_i-x_{i-1})^2}{\min_{x \in
     [x_{i-1},x_i)} p(x)} \, \Omega \, \norm{f}^2 = \frac{\delta ^2
   \Omega}{\pi ^2} \, \norm{f}^2 \,  
   \end{equation}
for $f\in PW_\Lambda (A_p)$, and \eqref{eq:c26} follows. 
\end{proof}

The fundamental inequality implies immediately a  sampling
theorem for $PW _\Lambda (A_p)$. 

\begin{thm}[Sampling inequality] \label{maxgap}
Let $\Lambda \subseteq [0,\Omega] \subseteq \br ^+_0$ and assume that
$\inf _{x\in  \br} p(x) >0 $. If 
\begin{equation}
  \label{eq:48}
\delta=\sup_{i\in \bz }  \frac{x_{i+1}-x_i} {\inf_{x \in
    [x_i,x_{i+1}]} \sqrt{p(x)}} < \frac{\pi }{\Omega ^{1/2}}\, ,
  \end{equation}
then, for all $f \in \PW \Lambda$,
\begin{equation}\label{eq:71}
\Big( 1- \frac{\delta {\Omega^{1/2}}}{\pi} \Big)^2 \, \norm{f}^2 \leq
\sum_{i \in \bz}\frac{x_{i+1}-x_{i-1}}{2} \abs{f(x_i)}^2 \leq \Big( 1+
\frac{\delta {\Omega^{1/2}}}{\pi} \Big)^2 \, \norm{f}^2 \,. 
\end{equation}
 If, in addition,  $\inf _{i\in \bz } (x_{i+1} - x_i) = \gamma >0$
 ($X$ is separated), then $X$ is a set of stable sampling for $\PW
 \Lambda$ with lower bound $  \gamma(1-\frac{\delta \Omega ^{1/2}}{\pi })^2
 $. Equivalently, the set of reproducing kernels $\{ k(x,\cdot
 ): x\in X\}$ is a frame for $PW_\Lambda (A_p)$. 
\end{thm}
\begin{proof}
 We use the triangle inequality and \eqref{eq:c25} to obtain 
 \begin{align*}
\sum _{i\in \bz } |f(x_i)|^2 \, \frac{x_{i+1}- x_{i-1}}{2} &= \norm{\sum _{i\in
    \bz} f(x_i ) \chi _i   }^2 \\
&\geq \Big(\norm{f} -     \norm{f-\sum_{i \in \bz} f(x_i) \chi_i }
\Big) ^2 \geq \Big(1 - \frac{\delta \Omega ^{1/2}}{\pi} \Big)^2 \,  \norm{f}^2
\, . 
 \end{align*}
The upper bound is already  in~\eqref{eq:c26}.
\end{proof}

Based on the sampling inequality \eqref{eq:71},  one may formulate several
 algorithms for the reconstruction of $f\in PW_\Lambda (A_p)$ from its
 samples $f(x_i), i\in \bz $. On the one hand one may use the manifold
 variations of the frame algorithm (iterative, accelerated iterations,
 or by means of a dual frame), on the other hand, one may use the
 following iterative algorithm from ~\cite{Groch92}. We set $P_\Lambda
 = c_\Lambda (A_p)$ for the orthogonal projection onto $PW_\Lambda
 (A_p)$.

\begin{thm} \label{thm:rec-alg}
Let $\Lambda \subseteq [0,\Omega] \subseteq \br ^+$ and assume that
$\inf _{x\in  \br} p(x) >0 $.  
Assume that  the sampling set $X$ satisfies the maximum gap condition \eqref{eq:48}.
Then  $f \in \PW \Lambda$ can be reconstructed from its  sampled  values $\big(f(x_i)\big)_{i \in \bz}$ 
by the following  algorithm.  

Initialization:  $  f_0 = h_0 =P_\Lambda \Bigl(\sum_{i \in \bz} f(x_i) \chi_i\Bigr) 
$.

Iteration: $h_{n+1}  = h_n- P_\Lambda \Bigl(\sum_{i \in \bz} h_n(x_i)
  \chi_i\Bigr)$ for $n\geq 0$.

Update: $f_{n+1}=\sum _{j=0}^{n+1} h_j = f_{n}+h_{n+1}$ for $n\geq
0$. 

Then 
\begin{equation}
  \label{eq:65}
  f= \lim _{n\to \infty } f_n = \sum_{n=0}^\infty h_n \,,
\end{equation}
with the error estimate 
\begin{equation}
  \label{eq:66}
  \norm{f-f_n} \leq \Big(\frac{\delta {\Omega^{1/2}}}{\pi}\Big)^{n+1}
  \frac{\pi+\delta {\Omega^{1/2}}}{\pi-\delta {\Omega^{1/2}}} \norm{f}
  \,. 
\end{equation}
\end{thm}
\begin{proof}
Define  $R$ by $Rf = P_\Lambda \Bigl(\sum_{i \in \bz}
f(x_i) \chi_i\Bigr)$. By the Plancherel-Polya
inequality~\eqref{eq:c26} the operator is bounded on $PW_\Lambda (A_p)$. With
this notation we have $h_0 = Rf$ and the iteration step is 
  \[ 
  h_{n+1} = (I- R) h_n = (I-R)^{n+1}h_0= (I-R)^{n+1} Rf\, . 
\] 
The sum  $ \sum h_n = \sum _{n=0}^\infty
(\mathrm{I}-R)^nRf$ is just the Neumann series for the inverse of $R$ applied to $Rf $. 
By the fundamental inequality~\eqref{eq:c25}  and the assumption on
$\delta $ the operator norm of
$\mathrm{I}-R$ on $PW_\Lambda (A_p)$    is bounded by   $\norm{\mathrm{I}-R} \leq \delta
\Omega ^{1/2} /\pi $. Since $\delta \Omega ^{1/2} /\pi <1$ by
assumption, the Neumann series converges in the operator norm  to
$R^{-1}$ and  consequently $f = \sum _{n=0}^\infty h_n$. The error
estimate follows from the properties of geometric series. 
\end{proof}

 \begin{rems}
(1)   If we replace the indicator functions in the reconstruction algorithm by partitions of unity with higher regularity with respect to $A_p$ then the convergence rate of the approximations can be increased. The results require an adapted form of Wirtinger's inequality and will be published elsewhere. \\
(2) The theorem requires only very weak regularity conditions for
$p$, essentially $p$ should be  bounded away from zero.  
\end{rems} 
%

\section{Landau's necessary density conditions }
\label{sec:landau}
 In this section we state and prove {density conditions}  in the style of Landau for
 sampling sequences $X \subset \br$, spectral sets $\Lambda \subseteq \br^+$ of finite Lebesgue measure, and functions in
 $PW_\Lambda(A_p)$.  We find necessary
 conditions on $X$ in terms of an appropriately defined version  of
 the 
 {Beurling density} such that the reproducing kernels
 $\set{k(x,\cdot) \colon x \in X}$ form either  a frame or  a Riesz
 sequence  for $PW_\Lambda(A_p)$. 
Recall that 
$\set{k(x,\cdot) \colon x \in X}$ is  a Riesz  sequence  for $PW_\Lambda(A_p)$, if  there are positive constants $C,D$ such that for all $c \in\ell^2(X)$
\[
C \sum_{x \in X}\abs{c_x}^2\leq \bignorm{\sum_{x\in X} c_x k(x,\cdot)}^2 \leq D \sum_{x \in X}\abs{c_x}^2 \,.
\]
Equivalently, for all $c\in \ell ^2(X)$ there exists an $f\in PW
_\Lambda (A_p)$, such that $f(x ) = c_x, \forall x\in X$,
therefore $X$ is also called a \emph{set of interpolation} for $\PW
\Lambda $. 

\subsection{Beurling density}
\label{sec:beurdens}
Assume $X$ is a\emph{relatively separated} subset  of $\br$,  i.e.,
$\max _{c\in \br} \# (X \cap[c,c+1]) = n_0 < \infty $.  This property
implies that an interval $I$  of 
length $|I|$ contains at most $(|I| +1)n_0$ points of $X$. The upper 
Beurling density  of $X$ is defined as 
\begin{equation*}
   D^+(X)= \varlimsup_{r \to \infty}  \sup_{ \abs I =r} \frac{\set{\#
       (X \cap I) \colon  I \subset \br \text{ closed interval}} }{r}
   \,, 
\end{equation*}
and the lower Beurling density is 
\begin{equation*}
   D^-(X)= \varliminf_{r \to \infty}  \inf_{\abs I =r} \frac{\set{\# (X \cap I) \colon  I \subset \br \text{ closed interval}} }{r} \,.
\end{equation*}
Landau's density conditions for  the classical
Paley-Wiener space $PW_\Lambda = \{f\in \Ltwo : \supp \, \hat{f}
\subseteq \Lambda \}$ state that  a set  of stable
sampling $X$ for $PW_\Lambda$  satisfies necessarily $D^-(X) \geq \abs \Lambda/(2
\pi)$. Similarly, if $X$ is a set of interpolation for $PW_\Lambda$
then $D^+(X) \leq \abs \Lambda/(2 \pi)$~\cite{Landau67a,Landau67}.  

What is the  the appropriate notion of density  for the Paley-Wiener
spaces of variable \bw ? Let us assume first that  $p$ is piecewise
constant, say $p(x) = p_k$ on  the interval $I_k$. According to
Proposition~\ref{prop:effbw} the function  $f\in PW_{[0,\Omega]}  (A_p)$  coincides  on $I_k$
with the restriction  of a  function in the Bernstein space $B_{\sqrt{\Omega/p_k}}$ to
$I_k$, so we expect the number of samples  in $I_k$ required  for the  reconstruction 
in  $I_k$  to be roughly 
\begin{equation}
  \label{eq:c27}
 \frac{\# (X \cap I_k)}{\abs {I_k}} \sim \sqrt {\frac{\Omega}{p_k}}
 \, .  
\end{equation}
Rewriting \eqref{eq:c27} as $\frac{\# (X \cap I_k)}{\abs {I_k}
  p_k^{-1/2}} \sim \Omega ^{1/2}$, we may interpret $\abs {I_k}
p_k^{-1/2}$ as a new measure (or distance function) on $\br$ and the
quantity in \eqref{eq:c27} as an average number of samples with
respect to this measure.

We therefore introduce the measure $\mu _p$ by 
\begin{equation}
  \label{eq:28}
  \mu_p(I)=\int_I p^{-1/2}(u) \, du \,.
\end{equation}
\begin{defn}
  Assume that  $p^{-1/2} \in L_{loc}^1$ and that  $X \subseteq \br$    is
  $\mu_p$-separated, i.e.,  $\inf \set{\mu_p([x,z]) \colon x, z \in X,x
    < z } >0$. The  upper  $A_p$-Beurling density is defined   as  
\begin{equation*}
   D_p^+(X)= \varlimsup_{r \to \infty}  \sup_{ \mu_p (I) =r}
   \frac{\set{\# (X \cap I) \colon  I \subset \br \text{ closed
         interval}}}{r} \,, 
\end{equation*}
and the lower $A_p$-Beurling density is 
\begin{equation*}
   D_p^-(X)= \varliminf_{r \to \infty}  \inf_{\mu_p (I) =r}
   \frac{\set{\# (X \cap I) \colon  I \subset \br \text{ closed
         interval } } }{r} \,. 
\end{equation*}
\end{defn}
Again, for $p\equiv 1$ these densities coincide with the standard
Beurling densities. 

\vspace{ 3 mm} 
%
%
%
To  derive \ndc\  for sampling and interpolation in $\PW \Lambda $, we
restrict our attention to the model case of eventually constant
$p$. From now on we  assume that $p$ satisfies the universal
assumption 
(\ref{eq:59}) and that 
\begin{equation} \label{eq:psmooth}
\begin{split} 
p, p' & \in AC_{loc}(\br) \,,  \\
 p(x)&=
 \begin{cases}
& p_- >0,  \text{ if } x <- R \\   
& p_+ >0,   \text{ if }  x > R \, . \end{cases}
 \end{split}
\end{equation}



\begin{thm}[Necessary density conditions for interpolation] \label{thm:landau-interpol}
Assume that $\Lambda \subseteq \br ^+$ has finite Lebesgue measure   and 
that  $p$ satisfies \eqref{eq:psmooth}.   If $\{k(x,\cdot) : x \in
X\}$ is a Riesz  sequence for $PW_\Lambda(A_p)$, then 
\[
D_p^+(X) \leq \frac{\abs {\Lambda^{1/2}}}{ \pi} \,.
\]
\end{thm}
\begin{thm}[Necessary density conditions for sampling]\label{thm:nds}
Assume that $\Lambda \subseteq \br ^+$ has finite Lebesgue measure and 
that  $p$ satisfies \eqref{eq:psmooth}.   If $\{k(x,\cdot) : x \in X\}$ is a frame for
$PW_\Lambda(A_p)$, then 
\[
D_p^-(X) \geq \frac{\abs {\Lambda^{1/2}}}{ \pi} \,.
\]
\end{thm}

Thus the quantity  $\frac{\abs {\Lambda^{1/2}}}{ \pi} $ is  the
critical density that separates  sets of stable sampling from sets of
interpolation.

\begin{rems}
1. We have seen in the introduction that for $p\equiv 1$, $A_p = -\frac{d^2}{dx^2}$, and $\Lambda =
[0,\Omega]$, the corresponding Paley-Wiener space
$PW_{[0,\Omega]}(A_p)$ is equal to the space of \bdl\ functions $\{
f\in \Ltwo : \supp \, \hat{f}\subseteq [-\Omega ^{1/2},
\Omega ^{1/2}]\}$. Then the necessary density condition is 
$D_p^-(X) = D^-(X) \geq \Omega ^{1/2}/\pi  =
|[-\Omega ^{1/2},\Omega ^{1/2}]|/(2\pi )$. Theorem~\ref{thm:nds} contains
Landau's result as a special case. The difference in formulation comes
from the use of a second order differential operator which identifies
positive and negative frequencies with a single spectral value.

2.  The assumption \eqref{eq:psmooth}   excludes both the toy example of  Section
\ref{sec:model-case} (because of lacking smoothness) and more general
parametrizing functions (when $p$ tends to $p_{\pm}$ at a certain
rate). To restrict the length of this paper, we will only treat the
case of eventually constant $p$ and return to weaker assumptions in
our future work. 


\end{rems}

We first compare the maximum gap condition of
Section~\ref{sec:irregular-sampling-} with the Beurling density. 
\begin{prop}
  \label{prop:recalgquasiopt}
Let   $X=\set{x_i \colon {i \in \bz}} \subseteq \br $ be a set  with $x_i <
x_{i+1}$ for all $i$ and $\lim _{i\to \pm \infty } x_i = 
\pm \infty $. If  
\begin{equation}\label{eq:c4}
  \sup_i \frac{x_{i+1}-x_i}{\inf_{x \in [x_i;x_{i+1}]} \sqrt{p(x)}} = \eta 
\end{equation}
then $D_p^-(X) \geq \eta ^{-1}$. 
\end{prop}
\begin{proof}
 The gap condition \eqref{eq:c4} implies that for all $i\in \bz $
$$
\int _{x_i} ^{x_{i+1}} \frac{1}{\sqrt{p(x)}} \, dx \leq \frac{x_{i+1}-x_i}{\inf_{x
    \in [x_i;x_{i+1}]} \sqrt{p(x)}} \leq \eta \, .
$$
Given a bounded, closed interval $I \subseteq \br $, let $i_0 = \min
\{i\in \bz : x_i \in I \}$ and $i_1 = \max
\{i\in \bz : x_i \in I\}$ the smallest  and largest indices of 
  $x_i \in I$. Then $I \subseteq [x_{i_0-1}, x_{i_1+1}]$ and 
  \begin{align*}
    \mu _p(I) &= \int _I \frac{1}{\sqrt{p(x)}} \, dx \leq \sum _{i=i_0
      -1}^{i_1} \int _{x_i} ^{x_{i+1}} \frac{1}{\sqrt{p(x)}} \, dx  \leq  \eta \Big(\# (X\cap I) +2\Big) \,  .
  \end{align*}
Consequently
$$
\frac{\# (X\cap I)}{\mu _p(I)} \geq \frac{1}{\eta} - \frac{2}{\mu
  _p(I)} \, ,
$$
and after taking a limit we obtain  $D_p^-(X) \geq \inv\eta  $. 
\end{proof}

Proposition~\ref{prop:recalgquasiopt} shows that condition~\eqref{eq:48} is sharp for
eventually constant para\-me\-trizing functions $p$. If $X$ is defined by
$x_0 = 0$ and $\frac{x_{i+1} - x_i}{\inf _{x\in [x_i, x_{i+1}]} \sqrt{p(x)}}
= \delta > \frac{\pi}{\Omega ^{1/2} -\epsilon }$, then 
$D_p^-(X) = (\Omega ^{1/2} -\epsilon ) /\pi  < \Omega ^{1/2}  /\pi
$ and thus cannot be a set of stable sampling by Theorem~\ref{thm:nds}.


The remainder of the paper is devoted to the proof of 
Theorems~\ref{thm:landau-interpol} and~\ref{thm:nds}. 
We follow the approach of   Nitzan and Olevskii~\cite{Nitzan12} who
compare a discrete set of reproducing kernels  to a \cont\ resolution
of the identity in the space of bandlimited functions.  Other
approaches, such as the original technique of
Landau~\cite{Landau67, Landau67a},  the technique of  Ortega-Cerd{\`a}
and Pridhnani~\cite{Ortega-cerda12}, or the approach of
\cite{Gr96landau,Gr08hap},  might be successful as well, but these
will require additional features, such as  the existence of a Riesz basis of
reproducing kernels. 
\subsection{Transformation to Schr\"odinger form}
\label{sec:transf-schr-form-1}
In the first step we transform the problem from the Sturm-Liouville
picture to the Schr\"odinger picture. This transformation enables us
to use the scattering theory of the Schr\"odinger operator. 
%
First we describe the unitary transform  that sends the operator $A_p$ to a Schr\"odinger operator.

Assume $p, p' \in AC_{loc}(\br)$, $p(x) >0$ for all $ x \in \br$. Define
\begin{equation}
\zeta(x) = \int_0^x p^{-1/2}(u) \, du \, , \label{eq:64}
\end{equation}
so that $\zeta  (x) = \mu _p([0,x])$ for $x>0$ and  $\zeta (x) = -
\mu _p([x,0])$ for $x<0$.  

\begin{prop} 
  \label{prop:liouville-transform}
Define  the \emph{Liouville transform} $U_L$  of $f\in L^2(\br)$ by  
$$
  U_Lf = (p^{1/4}f)\circ \inv \zeta \, .
$$
Then $U_L$ is a unitary operator on $\Ltwo $. It transforms the \saj\ operator $A_p$ to the \saj\ Schr\"odinger operator $B_q= -D^2 +q$ by conjugation:
\[
 B_q=U_L A_p U_L^* \,, \qquad  \mD(B_q)= U_L \mD(A_p) \, , 
\]
where  the potential $q$ of $B_q$ is 
\begin{equation}\label{eq:63}
q(\zeta(x))= - p(x)^{1/4} \bigl[p   \cdot  (p^{-1/4})' \bigr]'(x)
=\frac 1 4 p''(x) -\frac 1 {16} \frac{p'(x)^2}{p(x)}  \, . 
\end{equation}

In particular, if $\phi $ solves
 $(\tau _p -\lambda)\phi=0$, then $U_L \phi $ solves $(\tilde \tau _q-\lambda)\psi=0$.
\end{prop}
For a proof see ~\cite{Birkhoff69,Everitt74} or try a direct
computation. 
The next lemma explains the translation from the Sturm-Liouville
picture   to the Schr\"odinger picture in more detail. 

\begin{lem}  \label{prop:PWLiouville}
Assume $p, p' \in AC_{loc}(\br)$ and $p(x) >0$ for all $ x \in \br$. Then
\begin{enumerate}
\item \label{item:PWinvariant} $ U_L\big( PW_\Lambda(A_p)\big)=PW_\Lambda(B_q) $.
\item \label{item:9} $D^{\pm}_p(X)=D^{\pm}(\zeta(X))$ .
\item \label{item:10} Let   $k$ be  the reproducing kernel for
  $PW_\Lambda(A_p)$ and $h$  be  the reproducing kernel for $PW_\Lambda(B_q)$. Then
\begin{equation}
h(\zeta(x),\cdot)= p^{1/4}(x) U_L k(x,\cdot)\,.\label{eq:33}
\end{equation}
\item \label{item:11}  If $c\leq p(x) \leq C$  for all $x \in \br$,   then
  $(k(x,\cdot))_{x \in X}$ is a frame (Riesz sequence) for
  $PW_\Lambda(A_p)$ if and only if  $(h(\zeta(x),\cdot))_{x \in X}$ is
  a frame (Riesz sequence) for  $PW_\Lambda(B_q)$. 
\end{enumerate}
\end{lem}
\begin{proof}
 \ref{item:PWinvariant} Let 
$ c_\Lambda (A_p)$ be
 the spectral projection onto $PW_\Lambda(A_p)$ and  $f \in \PW
 \Lambda$. Then  
  \[
 c_\Lambda(A_p) f = f \text{  if and only if  } \big(U_Lc_\Lambda(A_p)
 U_L^*\big) U_Lf = U_Lf \, .
\]
Since  by spectral calculus  $ U_Lc_\Lambda(A_p) U_L^*  = c_\Lambda(B_q)$,
we see that $f\in PW_\Lambda (A_p)$ \fif\ $ U_L f  \in PW_\Lambda
(B_q)$. 

 \ref{item:9}  Observe that  for every interval $[a,b] =
 \zeta([\alpha,\beta])= [\zeta (\alpha), \zeta (\beta)]$ we have 
\[
\frac{ \# \big(\zeta(X) \cap [a,b] \big)}{b-a}
 =\frac{ \# \Big(\zeta(X) \cap \zeta\big([\alpha,\beta] \big) \Big)}{\zeta(\beta)-\zeta(\alpha)} =  \frac{ \# \big(X \cap [\alpha,\beta] \big)} {\mu_p([\alpha,\beta])} \,.
\]
Taking  limits on both sides, we find that $D^{\pm } (\zeta (X)) =
D_p ^{\pm }(X)$.

\ref{item:10} If   $f \in PW_\Omega(A_p)$, then  $U_Lf \in
PW_\Lambda(B_q)$ by (i), and 
\begin{equation}
  \label{eq:32}
  U_Lf (\zeta(x))=\inprod{U_Lf, h(\zeta(x),\cdot)}\, . 
\end{equation}
On the other hand 
\begin{align}
  U_Lf(\zeta(x))&=p^{1/4}(x) f(x) =p^{1/4}(x)\inprod{f, k(x,\cdot)}
  \notag \\
&=p^{1/4}(x)\inprod{U_Lf,U_L k(x,\cdot)}
=\inprod{U_Lf,p^{1/4}(x) U_L k(x,.)} \, .  \label{eq:c28}
\end{align}
The combination of  \eqref{eq:32} and \eqref{eq:c28}  yields~\eqref{eq:33}.

(iv)   The image of a frame (a Riesz  sequence) under an
invertible operator is again a frame (a Riesz  sequence).  So
$\{k(x,\cdot) :x \in X\}$ is a frame (a Riesz sequence) for $\PW \Lambda$, if and
only if $ \{U_Lk(x,\cdot) :x \in X\}$ is a frame (a Riesz sequence) for $\PW
[B_q] \Lambda$. Since $c\leq p(x)\leq C$,  $ \{U_Lk(x,\cdot) :
x \in X\}$ is a frame (a Riesz sequence), \fif\  $ \{p(x)^{1/4}
U_Lk(x,\cdot) : x \in X \}$ is a frame (a Riesz sequence).
\end{proof}

  From now on  we will work with the Schr\"odinger picture. By a
  slight abuse of notation we will denote the reproducing kernel for
  $\PW [B_q] \Lambda$ again by the symbol $k$. 

If $p$ is eventually constant, then by \eqref{eq:63} the potential $q$ has compact
support in some interval $[-a,a]$. We therefore assume that 
\begin{equation}
\tag {$MC_q$} \label{item:mcq} 
 \text{the potential $q$ is of the form  ~\eqref{eq:63} for some  $p$
   satisfying ~\eqref{eq:psmooth}.} 
\end{equation}
\\

 Lemma~\ref{prop:PWLiouville},\ref{item:10},\ref{item:11} implies an  equivalent formulation of the Theorems~\ref{thm:landau-interpol} and~\ref{thm:nds}. 

\begin{thm}[Necessary density conditions in $PW_\Lambda(B_q)$]\label{ncdschr}
Assume that  
$q$ satisfies~\eqref{item:mcq}. Let $k$ be the reproducing kernel for $PW_\Lambda(B_q)$.

(A) \emph{
  Interpolation:} 
  If  $\set{k(x,\cdot) \colon x \in X}$ is a Riesz  sequence in  $PW_\Lambda(B_q)$, then
\[
D^+(X) \leq \frac{\abs {\Lambda^{1/2}}}{ \pi} \,.
\]
(B) \emph{
  Sampling:} 
  If $\set{k(x,\cdot) \colon x \in X}$ is a frame for $PW_\Lambda(B_q)$, then
\[
D^-(X) \geq \frac{\abs{\Lambda^{1/2}} }{ \pi} \,.
\]
\end{thm}
Theorems~\ref{thm:landau-interpol} and \ref{thm:nds} now follow with the
translation lemma (Lemma~\ref{prop:PWLiouville}).  
The proof of  Theorem~\ref{ncdschr} 
will be carried out  in Section
\ref{sec:proofs-necess-dens} 

\subsection{Fundamental Lemmas}
\label{sec:fundamental-lemmas}

Most of the technical work for the proof of  Theorem~\ref{ncdschr} is
coded in some   lemmas on the localization and cancellation properties
of the reproducing kernel
for $\PW [B_q] \Lambda$. 
For the proofs we need information about the scattering theory of the Schr\"odinger operator.

For the spectral representation of the Schr\"odinger operator we
substitute the spectral parameter and  set  $\lambda=
\omega^2$. Thus if $\lambda $ is in the spectral set $\Lambda $, then
$\omega $ is in $\Lambda ^{1/2} = \{ \omega : \omega ^2\in \Lambda
\}$. This harmless, but convenient change of variables explains the
appearance of the set $\Lambda ^{1/2}$ in the formulation of the
density theorems. 

 \begin{prop}[{\cite[23.2]{Weidmann03},\cite[17.C]{Weidmann87}}]
\label{Schroescatt}
 If q satisfies \eqref{item:mcq}, then the eigenfunction equation
 $(\tilde \tau_q-\omega^2)\phi =0$ possesses a system of fundamental
 solutions of  the form 
   \begin{equation}
     \label{eq:fundsols}
     \Phi(\omega, x)= 
     \begin{cases}
       \begin{pmatrix}
         e^{i \omega x} +R_1(\omega) e^{-i \omega x} \\
         T(\omega) e^{-i \omega x}
       \end{pmatrix}
       & ,\qquad x < -a \\ 
       \begin{pmatrix} 
         T(\omega) e^{i \omega x} \\ 
         e^{-i \omega x} +R_2(\omega) e^{i \omega x}
       \end{pmatrix}
       & ,\qquad x > a
     \end{cases}
   \end{equation}
   The \emph{scattering matrix}
   \begin{equation}
     \label{eq:scatmat}
     \begin{pmatrix}
       T(\omega)& R_1(\omega) \\
       R_2(\omega)& T(\omega)
     \end{pmatrix}
   \end{equation}
   is unitary for all $\omega \in \br ^+$, 
  and  the  entries $T,R_1,R_2$  are holomorphic in $\omega$ for  $\omega \in \bc\setminus \br_0^-$.

The spectral measure of $\tilde \tau _q$ with respect to this
fundamental system is given by the matrix-valued Lebesgue measure  $
 \Big(
\begin{smallmatrix}
  d\omega & 0 \\ 0 & d\omega 
\end{smallmatrix}
\Big)
 = \mathrm{I_2} d\omega $. Consequently the 
  operator 
    \begin{equation}
     \label{eq:spm}
   \mF_{B_q} f (\omega)=\frac 1 {\sqrt{2 \pi}} \int_\br f(x) \overline{\Phi (\omega, x)} dx \,,
   \end{equation}
is unitary on $\Ltwo $
and diagonalizes $B_q$, i.e., 
\[
\mF_{B_q} B_q \inv{\mF_{B_q}} G (\omega)= \omega ^2  G(\omega) 
\]
for all $G \in L^2(\br,I_2 d\omega)$. 
The inverse of $\mF _{B_q}$ is 
   \begin{equation}
     \label{eq:ispm}
     \mF_{B_q}^{-1} G (x)=\frac 1 {\sqrt{2 \pi}} \int_{\br_0^+}  G(\omega)\cdot \Phi (\omega, x) d \omega \,
   \end{equation}
for  $G \in L^2(\br,I_2 d\omega)$. 
 \end{prop}


With this notation the reproducing kernel for $\PW [B_q] \Lambda$ is
simply 
\begin{equation}
  \label{eq:specfun}
  k(x, y)= k_\Lambda(x, y)=\frac 1 {2 \pi}\int_{\Lambda^{1/2}}
  \Phi(\omega, x) \cdot \overline{\Phi(\omega, y)} \, d \omega \,. 
\end{equation}
In this case it is obvious that the kernel exists for $\abs \Lambda
\leq \infty$ (see the remark after Proposition \ref{prop:basicpw}). 

The following three lemmas   describe several   properties
of the reproducing kernel. For the space  of \bdl\ functions $\{ f\in
\Ltwo : \supp \, \hat{f} \subseteq [-\Omega , \Omega ]\}$ the
reproducing kernel is $k(x,y) = \tfrac{\sin \Omega (x-y)}{x-y}$ and
the stated  estimates below are obvious.  For the Paley-Wiener
spaces of variable \bw\ they are highly non-trivial and, even in the model case $(MC_q)$,
they require the full power of the  scattering theory
(Proposition~\ref{Schroescatt}). 
 The following statements may also  be interpreted as 
  subtle cancellation properties of $k$. 


\begin{lem}[Weak localization] \label{lem:conthap}
 Assume that ~\eqref{item:mcq} holds and that $\Lambda $ is a 
 Borel set in $\br ^+_0$ with finite Lebesgue measure.   Let $k$ be the reproducing kernel for
 $PW_\Lambda(B_q)$. Then for every $\epsilon >0$ there is a constant
 $b_\epsilon$, such that
 \begin{equation}\label{eq:hapcont}
\sup _{x\in \br } \int_{\abs{y-x} >b_\epsilon}\abs{k(x,y)}^2 dy <
\epsilon ^2 \,.
\end{equation}
\end{lem}
\begin{lem}[Homogenous approximation property]\label{lem:hap}
 Assume that ~\eqref{item:mcq} holds and that $\Lambda $ is a bounded
 Borel set in $\br ^+_0$. Furthermore   assume that $X$ is a  set of stable
 sampling for $PW_\Lambda (B_q)$. Then for every $\epsilon >0$ there
 is a constant $b_\epsilon$, such that 
 \begin{equation}\label{eq:hapclass}
\sup _{y\in \br } \sum_{\substack{x \in X\\\abs{x-y} >b_\epsilon}}\abs{k(x,y)}^2  <
\epsilon ^2 \,.
\end{equation}
\end{lem}

The proof of these lemmas is deferred to Section \ref{sec:loc-hap-proof}.
\\ \\

\begin{lem}
  \label{lem: specfunasymp}
Assume that  $q$ satisfies  $(MC_q)$ with $\supp \, q \subseteq
[-a,a]$. Let 
$I =[\alpha,\beta]$ be a large,  closed interval. 
Then 
\begin{equation}
  \label{eq:c30}
  \bigabs{\frac 1 {\abs {I}}\int_{ {I}} k(y, y) \, dy - \frac{ \abs
      {\Lambda^{1/2}} } \pi} \leq \frac{2}{\sqrt \pi}\, \frac{\norm{R_1
    c_{\Lambda^{1/2}}}}{ \abs{I}^{1/2}} + \frac{1}{\abs I}
  \int_{-a}^ak(y,y) dy  +  \frac{2a}{\pi |I|} |\Lambda^{1/2} | \, . 
\end{equation}
  As a consequence, 
$$
\lim _{|I| \to \infty } \frac{1}{|I|} \int_{I} k(y, y) \, dy  = 
\frac{ \abs {\Lambda^{1/2}}}{\pi } \, . 
$$
\end{lem}
\begin{proof}
After substituting the fundamental solutions ~\eqref{eq:fundsols} into
\eqref{eq:specfun},  and using the unitarity of the
scattering  matrix~\eqref{eq:scatmat}, we obtain  for  $y > a$ 
\begin{equation} \label{eq:kbounded}
  \begin{split}
    k(y,y)=&\frac{1} {2 \pi}\int_{\Lambda^{1/2}} \big(\underset{=1}{\underbrace{\abs {T(\omega)}^2 + \abs
    {R_2(\omega)}^2}}
    + 1 + 2 \Re (R_2(\omega) e^{2 i  \omega y}) \big)\, d \omega \\
    =& \frac{1}{\pi}\big( \abs {\Lambda^{1/2}} +  \int_{\Lambda^{1/2}} \Re (R_2(\omega) e^{-2 i 
      \omega y})\, d \omega \big)\,.
  \end{split}
\end{equation}
 If $y < -a$, then similarly  
\[
k(y,y)= \frac{1}{\pi}\big( \abs {\Lambda^{1/2}} +  \int_{\Lambda^{1/2}} \Re (R_1(\omega) e^{- 2 i 
      \omega y})\, d \omega \big) \,.
\]
In the following we decompose the given interval $I$ into subintervals 
\[
I=[\alpha,\beta]= \big(I \cap (-\infty,-a)\big) \cup \big(I \cap [-a,a]\big) \cup \big(I \cap (a,\infty)\big)= I_1 \cup I_2 \cup I_3 \,.
\]
Without loss of generality we assume that $[-a,a] \subseteq I$ (If
$[-a,a] \not \subseteq I$, then some  of the integrals $\int _{I_k}
\dots $ are zero.)  Then
\begin{align*}
   \frac{1}{\abs {I}} \int_{I_3} k(y, y) \, dy &= \frac{|I_3|}{\pi
   \abs{I}}  \abs{\Lambda^{1/2}} + \frac 1 {\pi \abs {I}} \int_{I_3}
 \int_{\Lambda^{1/2}} \Re (R_2(\omega) e^{2 i      \omega y})\, d
 \omega \, dy \\ 
&  =\frac{|I_3|}{\pi \abs{I}}  \abs {\Lambda^{1/2}} + 
 \frac 1 {\pi \abs {I}} \Re \Big(\int_{\Lambda^{1/2}} R_2(\omega)  \int_{a}^{\beta}  e^{2 i  \omega y}\,dy \,d \omega \Big) \\
&  =\frac{|I_3|}{\pi \abs{I}}  \abs {\Lambda^{1/2}} + \frac 1 {\pi \abs {I}}
\Re \Big(\int_{\Lambda^{1/2}} R_2(\omega)  e^{  i \omega(a+
  \beta)} \frac{\sin (\omega (\beta-a))}{\omega } d \omega \Big)
\,, 
\end{align*}
and, with the substitution $u = \omega (\beta -a)$, 
\begin{align*}
 \frac 1 {\pi \abs {I}}\bigabs{
  \int_{\Lambda^{1/2}} R_2(\omega) e^{  i \omega(a+ \beta)}
  \frac{\sin (\omega (\beta-a))}{\omega}  d \omega}
   \leq &\frac { \norm{R_2 c_{\Lambda^{1/2}}}} {\pi \abs {I}} \Big((\beta-a)\int_{(\beta-a){\Lambda^{1/2}}}\frac{\sin^2 u}{u^2} du \Big)^{1/2} \\
 \leq& \frac{\norm{R_2 c_{\Lambda^{1/2}}}}{\sqrt \pi}{\abs {I}^{-1/2}} \,. 
\end{align*}
 By a similar calculation the contribution of $I_1$ yields 
\[
 \frac 1 {\abs {I}} \int_{ {I_1}} k(y, y) \, dy = \frac{|I_1|}{\pi \abs{I}}  \abs {\Lambda^{1/2}} + 
\frac 1 {\pi \abs {I}} \int_{I_1} \int_{\Lambda^{1/2}} \Re (R_1(\omega) e^{ -2 i      \omega y})\, d \omega \, dy 
\]
and
\[
\frac 1 {\pi \abs {I}} \bigabs{\int_{I_1} \int_{\Lambda^{1/2}} \Re (R_1(\omega) e^{ -2 i      \omega y})\, d \omega \, dy}
  \leq \frac{\norm{R_1 c_{\Lambda^{1/2}}}}{\sqrt \pi} \abs{I}^{-1/2} \,.
\]
After summing these contributions,   we obtain  
\begin{align}
\bigabs{\frac 1 {\abs {I}} \int_{ {I}} k(y, y) \, dy &-\frac{\abs {\Lambda^{1/2}}}{\pi}} \leq
\bigabs{  \frac 1  {\abs {I}} \int_{ {I_1 \cup I_3}} k(y, y) \, dy -
  \frac{|\Lambda ^{1/2}|}{\pi}} +
   \frac 1  {\abs {I}} \bigabs{\int_{ [-a,a]} k(y, y) \, dy} \notag \\
&\leq  \frac{1}{\sqrt \pi \abs{I}^{1/2}} \Big( \norm{R_1 c_{
    \Lambda^{1/2}}} + \norm{R_2 c_{\Lambda^{1/2}}} \Big)+ \frac{1}{\abs I}
\int_{-a}^a \abs{k(y,y)} dy  +  \frac{2a}{\pi |I|} |\Lambda^{1/2} | \, .
\end{align}
As the matrix \eqref{eq:scatmat} is unitary,   $\abs{R_1(\omega)}=\abs{R_2(\omega)}$ for all $\omega \geq 0$, and
 $\norm{R_1 c_{
    \Lambda^{1/2}}}=\norm{R_2 c_{\Lambda^{1/2}}}$, and~\eqref{eq:c30} follows. 
%
\end{proof}

From \eqref{eq:kbounded} and the continuity of $k$  we extract a crucial property of the reproducing
kernel $k$. 
\begin{cor} \label{cor:kbound}
If  $q$ satisfies condition  $(MC_q)$, then the diagonal of the kernel
$k$ is uniformly bounded:
$$
\sup _{y\in \br } k(y,y) = C_{RK} <\infty \, .
$$ 
\end{cor}

In the following lemma we gather some facts about  frames and Riesz
sequences  in a
reproducing kernel Hilbert space $H$. 
\begin{lem}\label{lem:RB-kernel}
Let $H$ be a reproducing kernel Hilbert space with  kernel $k$. 
Assume that $\{ k(x,\cdot ): x\in X\}$ is a frame for $H$ with
canonical dual frame $\{g_x: x\in X\}$. 
Then $k$ and $g_x$ satisfy the following inequalities:
\begin{align}
 &  \sum_{x \in X} k(x,y) \overline{g_x(y)} = k(y,y)  \label{eq:c41} \\
& \sum_{x \in X}\abs{g_x(y)}^2 \leq C k(y,y)  \label{eq:c42}  \\
& \sup_{x \in X}\abs{\inprod{k(x,\cdot), g_x}} \leq 1 \label{eq:c43} \\
 & \sup _{x\in X} \norm{g_x} = C <\infty \, .  \label{eq:c44}
\end{align}
If $\{ k(x,\cdot ): x\in X\}$ is a Riesz sequence for a subspace
$V\subseteq H$ with biorthogonal basis  $\{g_x: x\in X\}\subseteq V$,
then \eqref{eq:c41} is replaced by the inequality
\begin{equation}
  \label{eq:c45}
    \sum_{x \in X} k(x,y) \overline{g_x(y)} \leq k(y,y)  \, . 
\end{equation}
Instead of \eqref{eq:c44} holds equality $\abs{\inprod{k(x,\cdot),
    g_x}}=1 $ for all $x\in X$. 
\end{lem}



\begin{proof}
The inequality~\eqref{eq:c45} follows from 
 \begin{align*}
    \sum_{x \in X} k(x,y) \overline{g_x(y)} 
&= \sum_{x \in X}\inprod{k(x,\cdot),k(y,\cdot)}\inprod{k(y,\cdot),g_x}\\
&= \biginprod{\sum_{x \in X} \inprod{k(y,\cdot),g_x}\, k(x,\cdot),k(y,\cdot)}
= \inprod{P_V k(y,\cdot),k(y,\cdot)} \\
& \leq \| k(y,\cdot )\|^2 = k(y,y) \,.
  \end{align*}
The proof for frames  is the same (just omit  the projection).

Item \eqref{eq:c42}
 follows from 
\[
\sum_{x \in X}\abs{g_x(y)}^2=\sum_{x \in X}\abs{\inprod{g_x, k(y,\cdot)}}^2\leq C\norm{k(y, \cdot)}^2 = C k(y,y)\,,
\]
where $C$ is the upper frame bound for $\set{g_x \colon x \in X}$.

Item \eqref{eq:c43}
 is an immediate consequence of the minimal $\ell^2$-norms of the
 coefficients in the canonical  frame expansion~\cite{DS52}:
\[
k({x'},\cdot)=\sum_{x \in X}\inprod{k({x'},\cdot),g_x} k(x,\cdot )  =
1 \cdot   k({x'},\cdot) \quad \text{ for every } x' \in X\,, 
\]
so
\[
\abs{\inprod{k({x'},\cdot),g_{x'}}}^2 \leq \sum_{x \in
  X}\abs{\inprod{k({x'},\cdot),g_{x}}}^2 \leq 1\quad \text{ for every
} x' \in X\,. 
\]
Finally \eqref{eq:c44} is a general fact about frames. 
\end{proof}

\subsection{Proof of Theorem \ref{ncdschr}}
\label{sec:proofs-necess-dens}

It is easy to see that every set of interpolation for  $PW_\Lambda
(B_q)$ must be  separated and that every set of stable sampling for  $PW_\Lambda
(B_q)$ must be relatively separated.


\begin{proof}[(A) Proof of the  necessary density conditions  for interpolation]
Assume that $\set{k(x,\cdot) \colon x \in X}$ 
is a Riesz sequence  in $PW_\Lambda 
(B_q)$ with biorthogonal basis $\set{g_x \colon x \in X}$. 
  For every closed,  bounded interval $I=[\alpha,\beta] \in \br$
let $V$ be  the (finite-dimensional) subspace $V= V_I= \spann
  \set{k(x,\cdot) \colon x \in X \cap I}$ and $P_V$ be   the orthogonal projection
 from  $PW_\Lambda(A)$ onto $V$. 
  Then   $\set{P_V g_x \colon x \in X \cap I}\subseteq V$ is the biorthogonal  basis to
  $\set{k(x,\cdot) \colon x \in X \cap I}$ and $\|P_V g_x \| \leq \|g_x\|
  \leq C$ for all $x$.  By~\eqref{eq:c45} 
\[
\sum_{x \in X \cap I} k(x,y) \overline{P_V g_x (y)} \leq
k(y,y)
\, .
\]
We now integrate both sides over a suitably enlarged interval
$I_b=[\alpha-b, \beta + b]$ and obtain 
\begin{equation}
  \label{eq:ndi1}
  \int_{I_b} \sum_{x \in X \cap I} k(x,y) \overline{P_V g_x (y)} d y
  \leq \int_{I_b} k(y,y) dy \, .
\end{equation}
 The biorthogonality   $\inprod{  k(x,\cdot) ,P_V
  g_x}  = \inprod{  k(x,\cdot) , g_x}=1$ for  $x\in X\cap I$ implies
that 
\[
\int_{I_b} k(x,y) \overline{P_V g_x (y)} d x = 1- \int_{\br \setminus
  {I_b}} k(x,y) \overline{P_V g_x (y)}  d y \, .
\]
Now fix   $\epsilon >0$ and let $b=b_\epsilon $ be the weak localization
constant of Lemma~\ref{lem:conthap}. 
If  $x \in X
\cap I$ and $y \in \br \setminus I_b$,  then $\abs{x-y} >
b$. Therefore  Lemma~\ref{lem:conthap}  implies that 
\begin{equation}
  \bigabs{\int_{\br \setminus I_b} k(x,y) \overline{P_V g_x (y)}  d
    y}^2 \leq \| P_V k(x,\cdot )\| ^2  \int_{\br \setminus I_b}\abs{ k(x,y)}^2  d y  
\leq C^2 \epsilon^2 \,,
\end{equation}
and consequently
\[
\Big| \sum _{x\in X\cap I} \int_{I_b} k(x,y) \overline{P_V g_x (y)} d
y \Big| = \Big|  \sum _{x\in X\cap I} \Big( 1-  \int_{\br \setminus I_b} k(x,y) \overline{P_V g_x (y)} d
y\Big) \Big|\geq \# (X\cap I) \, (1- C \epsilon ) \,.
\]
Inserting this estimate in the  left-hand side of  (\ref{eq:ndi1})
yields the lower estimate
$$
\frac{1}{\abs I}\int_{I_b} k(y,y) dy \geq (1- C \epsilon) \frac{\# (X
  \cap I)}{\abs I} \, , 
$$
whereas   Lemma \ref{lem: specfunasymp} leads to the upper estimate
$$
\frac{1}{\abs I}\int_{I_b} k(y,y) dy \leq \frac{\abs{I_b}}{\abs I}
\Bigg[\frac{\abs {\Lambda^{1/2}}} {\pi} + C'\abs {I_b} ^{-1/2} \Bigg]
\, .
$$
Since $\abs{I_b}/\abs I = 1 + 2b/\abs{I}$, we  obtain 
$$
 \frac{\# (X \cap I)}{\abs I} \leq (1- C \epsilon) ^{-1}
 \Bigg[\frac{\abs {\Lambda^{1/2}}} {\pi} + \frac{C'}{\abs {I } ^{1/2}} \Bigg] 
$$

We take the supremum over all intervals with $\abs I =r$ and then the
limit  $r \to \infty$, this yields $D^+(X) \leq (1- C \epsilon) ^{-1}
 \tfrac{\abs {\Lambda^{1/2}}} {\pi}$.  As $\epsilon >0$ was arbitrary,
 we have proved
that $D^+(X) \leq \frac{\abs{\Lambda^{1/2}}}{\pi}$. 
\end{proof}

\begin{proof}[(B) Proof of the necessary density conditions for sampling]
We assume that \\  $\set{k(x,\cdot) \colon x \in X}$ is a frame for
$PW_\Lambda (B_q)$ with  canonical dual frame  
 $\set{ g_x \colon x \in X}$. 

Since we will use Lemma~\ref{lem:hap},  we first assume that the
spectral set $\Lambda $ is bounded. 

Let $\epsilon >0$ be given. 
This time we use Lemma~\ref{lem:RB-kernel} \eqref{eq:c41} to write
$k(y,y) = \sum_{x \in X} k(x, y) \overline{g_x(y)}$ for all $y\in \br
$ and use this expression to estimate the averaged kernel in Lemma~\ref{lem:
   specfunasymp}. By Lemma~\ref{lem:
   specfunasymp} there exists an 
 $r_0= r_0(\epsilon) $ such that for all intervals $I=[\alpha,\beta]$
 of length $r>r_0$ 
$$
 \Big(\frac{ \abs {\Lambda^{1/2}} } \pi-\epsilon \Big)\abs I \leq  \int_I
 k(y,y) dy \, .
$$ 
The combination of these facts leads to 
\begin{equation}\label{eq:31}
 \Big(\frac{ \abs {\Lambda^{1/2}} } \pi-\epsilon \Big)\abs I \leq  \int_I
 k(y,y) dy =\bigabs{\int_I\sum_{x \in X} k(x, y) \overline{g_x(y)}
   dy}\, . 
\end{equation}
 Let $b= b_\epsilon $ be
larger than both  constants $b_\epsilon$ from
Lemmas~\ref{lem:conthap}  and ~\ref{lem:hap} and set
$I_-=[\alpha+b,\beta-b]$ and  $I_+=[\alpha-b,\beta+b]$.  We partition
$X$ accordingly  and write
\begin{align*}
\sum_{x \in X} k(x, y) \overline{g_x(y)}=&\Big(\sum_{x \in X\cap I_{-}}+\sum_{x \in X \cap( \br\setminus I_+)}+\sum_{x \in X \cap(I_+ \setminus I_-)} \Big) k(x, y) \overline{g_x(y)}\\
 =& A_1(y)+A_2(y)+ A_3(y)
\end{align*}

\emph{Estimate of $\int _I A_2$.}  Note that $y\in I$ and $x \in X
\setminus I_+$ implies that $|x-y|>b$. Lemma~\ref{lem:hap} asserts
that $\sum _{x\in X \setminus I_+} |k(x,y)|^2 < \epsilon
^2$. Consequently, using also \eqref{eq:c42}, we obtain
\begin{align*}
  \bigabs{\int_I A_2(y)d y} &\leq \int_I \big(\sum_{x \in X \cap(
    \br\setminus I_+)}\abs{ k(x, y) }^2\big)^{1/2}\big( \sum_{x \in
    X}\abs{g_x(y)}^2\big)^{1/2} dy \\ 
&\leq C\epsilon \sup_{y \in \br} k(y,y)^{1/2} \abs I \, . 
\end{align*}
Since the diagonal of the kernel $k$ is uniformly bounded by
Corollary~\ref{cor:kbound}, the final estimate for $A_2$ is 
\begin{equation}
  \label{eq:a2}
\bigabs{  \int _I A_2(y) \, dy } \leq C_2 \epsilon |I| \, .
\end{equation}

\emph{Estimate of $\int _I A_3$.}
For the third term observe that 
\begin{align}
\int_I \abs{A_3(y)} d y 
&\leq \sum_{x \in X \cap( I_+\setminus I_-)} \int_\br \abs{ k(x, y) }\, \abs{g_x(y)}dy 
\leq  \sum_{x \in X \cap( I_+\setminus I_-)} \norm { k(x, \cdot) }
\norm{g_x}  \, . 
\end{align}
Since $X$ is relatively  separated with covering constant $n_0 = \max
_{c\in \br } \# (X \cap [c,c+1]) $, this sum contains at most  $(|I_+
\setminus I_-| +1) n_0 \leq   (2b+1) n_0 $ terms. Using the boundedness of the diagonal of $k$ from
Corollary~\ref{cor:kbound}  and of
the canonical dual frame, the final estimate for $A_3$ is
\begin{equation}
  \label{eq:a3}
\int_I \abs{A_3(y)} \, dy \leq  C C_{RK} (2b+1) n_0  = C_3 \, . 
  \end{equation}
Here $C_3$ is a constant depending on $\epsilon $ via $b = b_\epsilon
$, but $C_3$ is independent of $I$. 


\emph{Estimate of $\int _I A_1$.} 
Next  we estimate $\abs{\int_I A_1(y) dy}$. Since  $\int_I=
\int_\br-\int_{\br\setminus I}$ and $|\langle k(x,\cdot ), g_x\rangle
| \leq 1$ by Lemma~\ref{lem:RB-kernel}, we obtain 
\begin{align}
\bigabs{\int_I k(x, y) \overline{g_x(y) }dy} & \leq 
\bigabs{\int_\br k(x, y) \overline{g_x(y)} dy}
+\bigabs{\int_{\br\setminus I} k(x, y) \overline{g_x(y)}dy} \\
&\leq 1 + \Big( \int _{\br \setminus I} |k(x,y)|^2 \, dy \Big)^{1/2}
\|g_x\| \, .
  \end{align} 
From   $x \in X \cap I_-$ and $y \in \br\setminus I$ it follows that
$\abs{x-y} > b$, therefore  by Lemma~\ref{lem:conthap} a single term
contributing to $A_1$ is majorized by 
\[
\bigabs{\int_{ I} k(x, y) \overline{g_x(y)}dy} \leq 1 + C_1 \epsilon
\, .
\]
This estimate  implies
\begin{align} \label{eq:a1}
 \bigabs {\int_I A_1(y) dy} \leq \sum_{x \in X\cap I_{-}}\bigabs{\int_I k(x, y) \overline{g_x(y)} dy} 
\leq  (1+ C_1 \epsilon) \,  \# (X \cap I_-) \, . 
\end{align}
Combining the estimates for $A_1, A_2, A_3$,  we obtain
\begin{align*}
  \Big(\frac{\abs {\Lambda^{1/2}}}\pi-\epsilon\Big )\abs I 
&\leq \bigabs{\int_IA_1(y)    dy}+ \bigabs{\int_I A_2(y) dy}+ \bigabs{\int_I A_3(y) dy} \\
&\leq(1+ C_1 \epsilon) \#(X \cap I) + C_2 \epsilon \abs I + C_3 \, . 
  \end{align*}
Therefore 
\begin{equation}
  \label{eq:bdsampest}
\frac{\#(X \cap I)}{\abs I} \geq (1+C_1 \epsilon )^{-1} \Bigg(
\frac{\abs {\Lambda^{1/2}}}\pi-\epsilon - C_2 \epsilon  - \frac{C_3
}{ |I|} \Bigg) \, .
\end{equation}
Now take the infimum over $\abs I =r$ and let  $r$ tend to
$\infty$. Again, since  $\epsilon >0$ is  arbitrary, we conclude that 
  $D^-(X) \geq \frac{\abs {\Lambda^{1/2}}}{\pi}$. 

So far we have proved the necessary density condition $D^- (X) \geq
\frac{\abs {\Lambda^{1/2}}}{\pi}$ under the assumption that the
spectral set $\Lambda $ is bounded. 

Now let $\Lambda \subseteq \br ^+$ be an arbitrary set of finite
measure and assume that $X$ is a set of stable sampling for
$PW_\Lambda (B_q)$.  Let   $\Omega >0$. Then $\Lambda \cap [0,\Omega]
$ is bounded and the Paley-Wiener space $PW_{\Lambda \cap
  [0,\Omega]}(B_q)$ is a closed subspace of $PW_\Lambda (B_q)$.  In
particular, every set of stable sampling for $PW_\Lambda (B_q)$ is a
set of stable sampling for $PW_{\Lambda \cap [0,\Omega ]} (B_q)$. 
From the main part of the proof we conclude that 
$$
D^-(X) \geq \frac{|\Lambda ^{1/2} \cap [0,\Omega ^{1/2}]|
}{\pi } \, .
$$
Since $\Omega >0$ was arbitrary, we conclude that $D^-(X) \geq
|\Lambda ^{1/2}|/\pi $. Thus this necessary condition holds for
arbitrary spectral sets of finite measure. 
\end{proof}

\section{Localization and Cancellation Properties of the Reproducing Kernel}
\label{sec:loc-hap-proof}

In this section we prove the decisive   Lemmas~\ref{lem:conthap} and
\ref{lem:hap}. 

\begin{proof}[Proof of weak localization ---  Lemma~\ref{lem:conthap}]
Since $|k(x,y)| = |k(y,x)|$, we will  show that there exists a $b= b_\epsilon $ such that 
$\int_{\abs{x-y} >b}\abs{k(x,y)}^2 \, dx   <
\epsilon ^2$ for all $y\in \br $. 

  We distinguish several  cases.

\textbf{Case I: $\pmb{\abs y \leq a}$.} Since $|x-y| \geq |x| - |y|
\geq |x| -a$, it  suffices to show  that there is a constant $b$, such
that  $\int_{\abs{x} >b}\abs{k(x,y)}^2 dx < \epsilon$. 
We use a compactness argument for this case. 

We first  verify that $y \mapsto k(\cdot,y)$ is a \cont\ mapping from
$[-a,a]$ to $L^2(\br)$, so  the set  $\set{k(\cdot,y) \colon \abs y \leq
  a}$ is compact in $L^2(\br)$. Then by  the Kolmogorov-Riesz theorem
(e.g.,\cite{Holden10, Weil41}) 
there is a constant $b_\epsilon$ such that 
\[
\norm{k(\cdot,y)\, c_{\abs{\cdot} > b_\epsilon}} < \epsilon \text{
  for all   } y \in [-a,a] \, .
\] 
To  verify the continuity of $y \mapsto k(\cdot, y)$ we use the dual
characterization of the norm as follows:
\begin{align*}
  \norm{k(\cdot, y)-k(\cdot,y')}= &\sup \set{
    \abs{\inprod{k(\cdot,y)-k(\cdot, y'),f}} \colon f \in
    PW_\Lambda(B_q), \norm{f} \leq 1}\\ 
=& \sup \set{ \abs{{f(y)-f(y')}} \colon f \in PW_\Lambda(B_q), \norm{f} \leq 1} \,,
\end{align*}
Using the representation formula of Lemma~\ref{prop:charbw}, we obtain 
\begin{align*}
  \abs{f(y)-f(y')} = &\bigabs{\int_{\Lambda^{1/2}} \mF_{B_q} f (\omega)\cdot [\Phi(\omega,y)-\Phi(\omega,y')] d \omega} \\
\leq & \int_{\Lambda^{1/2}} \abs{\mF_{B_q} f (\omega)} \abs{\Phi(\omega,y)-\Phi(\omega,y')}  d\omega \,.
\end{align*}
\\
An argument of Stolz~\cite[Thm.6]{Stolz92} (see
also~\cite[3.1]{Teschl08}) asserts  that $\Phi$
is uniformly bounded on  $[-a,a] \times \br^+$. (Actually, the theorem
states  the  boundedness of $\Phi(\lambda,\cdot)$, but  the proof verifies
boundedness in the spectral variable as well.)  

Set $C_\Phi =\sup_{x \in \br, \lambda \in \br^+} \abs {\Phi(\lambda,
  x)}$ and 
choose $\epsilon >0$. Then there is a number $u>0$ such that $\abs {\Lambda^{1/2} \setminus [0,u]} < \big(\frac{\epsilon}{2 C_\Phi}\big)^2$. 
 Since  $\Phi$ is uniformly \cont\ on the closure of $[-a,a] \times
{\Lambda^{1/2}\cap [0,u]}$, we can choose $\delta >0$ such  that
$\abs{\Phi(\omega,y)-\Phi(\omega,y')}< \epsilon$ for all $\omega \in
{\Lambda^{1/2} \cap [0,u]}$ and $\abs{y-y'}< \delta $. Then 
\begin{align*}
 \int_{\Lambda^{1/2} \cap[0,u]} \abs{\mF_{B_q} f (\omega)}
 \abs{\Phi(\omega,y) &-\Phi(\omega,y')}  d\omega \leq \epsilon
 \norm{\mF_{B_q} f}_{L^1({\Lambda^{1/2}})} \\
& \leq \epsilon
 \,C_{\Lambda^{1/2}} \norm{\mF_{B_q}
   f}_{L^2({\Lambda^{1/2}})}=\epsilon \,C_{\Lambda^{1/2}} \norm{
   f}_{L^2(\br)}   \, ,
\end{align*}
 because  $L^2({\Lambda^{1/2}}, I_2 d\omega)\subseteq L^1({\Lambda^{1/2}},
I_2 d\omega)$ and  the spectral transform is unitary. 
On the other hand
\begin{align*}
\int_{\Lambda^{1/2} \setminus[0,u]} \abs{\mF_{B_q} f (\omega)} \abs{\Phi(\omega,y)-\Phi(\omega,y')}  d\omega &\leq 2 C_\Phi \int_{\Lambda^{1/2} \setminus[0,u]} \abs{\mF_{B_q} f (\omega)}   d\omega \\
&\leq  2 C_\Phi \norm{\mF_{B_q} f }_{L^2({\Lambda^{1/2}})} \big(\abs {\Lambda^{1/2} \setminus [0,u]}\big)^{1/2}\\ & \leq  \epsilon  \norm{\mF_{B_q} f } _{L^2({\Lambda^{1/2}})} \,,
\end{align*}
so we obtain for $|y-y'|<\delta $
\[
 \abs{f(y)-f(y')}< C \epsilon \norm{f} \,.
\]
Taking the supremum over all $f$ in the unit ball of $\PW [B_p] {\Lambda}$ we obtain the desired continuity of $y \mapsto k(y, \cdot)$.\\

\textbf{Case II: $\pmb{ y > a}$.} We split the integral into three 
parts as follows:
\begin{align*}
\int_{\abs{x-y} >b}\abs{k(x,y)}^2 dx &= \int_{\substack{\abs{x-y} >b \\
    \abs x \leq a}}\abs{k(x,y)}^2 dx + \int_{\substack{\abs{x-y} >b \\
     x > a}}\abs{k(x,y)}^2 dx \\
& +\int_{\substack{\abs{x-y} >b \\     x <- a}}\abs{k(x,y)}^2 dx  = A + B+C\,, 
  \end{align*}
and estimate each  integral 
separately.

To estimate $A$, 
it is sufficient to find a value $b_0$ large enough, such that
\begin{equation}
  \label{eq:c32}
 \int_{ {\abs x \leq a}}\abs{k(x,y)}^2 dx < \epsilon \qquad \text{for
   all}\quad \abs y \geq  b_0\,.  
\end{equation}
By a straightforward calculation 
\begin{align*}
  \int_{ {\abs x \leq a}}\abs{k(x,y)}^2 dx\ =&
  \int_{ {\abs x \leq a}}\Biggl( \int_{\Lambda^{1/2}} \Phi(\omega,x)\cdot \overline{\Phi(\omega,y) }d \omega \,\,  \overline{\int_{\Lambda^{1/2}} \Phi(\mu,x)\cdot \overline{\Phi(\mu,y) }d \mu} \Biggr)\,  dx\ \\
=& \sum_{i,k=1}^2 \iint_{{\Lambda^{1/2}} \times {\Lambda^{1/2}}}\Biggl(\int_{\abs x \leq a} \Phi_i(\omega,x) \overline{\Phi_k(\mu,x)} \, dx \Biggr)\;\cdot\; \overline{\Phi_i(\omega,y)} \Phi_k(\mu,y) \,d\mu \, d\omega\\
=& \sum_{i,k=1}^2 \iint_{{\Lambda^{1/2}} \times
  {\Lambda^{1/2}}}\Psi_{i,k}(\omega,\mu)\;\cdot\;
\overline{\Phi_i(\omega,y)} \Phi_k(\mu,y) \,d\mu \, d\omega \,. 
\end{align*}
Here  the functions $\Psi_{i,k}$  are \cont\ in $\omega $ and $\mu $.  
By ~\eqref{eq:fundsols} for $|y| \geq a$, the products $ \overline{\Phi_i(\omega,y)} \Phi_k(\mu,y)$ are  of the form 
\begin{equation}
\alpha_{ik}(\omega,\mu) e^{\pm i(\omega-\mu)y}+\beta_{ik}(\omega,\mu) e^{\pm i(\omega+\mu)y} \label{eq:18}
\end{equation}
with smooth coefficients $\alpha _{ik}, \beta _{ik}$. 
Consequently, we may interpret the map $y \to \int _{|x| \leq a}
|k(x,y)|^2 \, dy$ as a sum of two-dimensional Fourier transforms  of
continuous functions on $\Lambda^{1/2 } \times {\Lambda^{1/2} }$. By
the Riemann-Lebesgue Lemma  $\lim _{y\to \infty } \int _{|x| \leq a}
|k(x,y)|^2 \, dy  = 0$ and \eqref{eq:c32} is proved.


To estimate the term $B$, 
we first   obtain an  explicit expression  for $k(x,y)$ in terms of
the scattering coefficients. Since  $x,y > a$, the scattering
relations  \eqref{eq:fundsols} yield  
\begin{align*}
  \Phi(\omega,x) \cdot \overline{\Phi(\omega,y)}=&
e^{i \omega (x-y)}\Bigl( 
\underset{=1}{\underbrace{\abs{T(\omega)}^2+\abs{R_2(\omega)}^2}}
+e^{- 2 i \omega (x-y)} + R_2(\omega) e^{2 i \omega y} + \overline{R_2(\omega)} e^{- 2 i \omega x}\Bigr)\\
=& e^{i \omega (x-y)} + e^{-i \omega (x-y)} +R_2(\omega)e^{i \omega
  (x+y)}+\overline{R_2(\omega)}e^{-i \omega (x+y)} \, . 
\end{align*}
After integrating the last expression over ${\Lambda^{1/2} }$, we obtain
\begin{equation}\label{eq:specfunas}
k(x,y) 
= \mF(c_{\Lambda^{1/2}})(y-x)+\mF(c_{\Lambda^{1/2}})(x-y)+ \mF(R_2
\,c_{\Lambda^{1/2}})(-x-y)+\mF(\overline R_2\, c_{\Lambda^{1/2}})(x+y)
\, . 
\end{equation}
Since $c_{\Lambda^{1/2} }$ and $R_2c_{\Lambda^{1/2} }$ are in
$L^2(\br )$, so are their Fourier transforms. Thus there exists a
constant $b_1$, such that 
$$
\int _{|z| > b_1} \Big( |\mF c_{\Lambda^{1/2} } (z) |^2 +  |\mF (R_2
c_{\Lambda^{1/2} }) (z) |^2 \Big) \, dz < \epsilon ^2 \, .
$$
Consequently for $x+y > |y-x| \geq b_1$ and $|x| \geq a$,   we obtain 
\[
B= \int_{\abs{y-x}> b_1, |x| >a}\abs{k(x,y)}^2 dx <\epsilon ^2 \, .
\]
To estimate $C$, where   $x< -a $ and $y>a$,  we use  again  the unitarity  of the
scattering matrix and obtain 
\begin{equation}\label{eq:5}
   \Phi(\omega,x) \cdot \overline{\Phi(\omega,y)}=
e^{i \omega(x-y)}\overline{T(\omega)} +e^{-i \omega(x-y)}{T(\omega)}+
e^{-i \omega(x+y)}
\underset{=0}{\underbrace{
\Bigl(R_1(\omega)\overline{T(\omega)} +T(\omega)\overline{R_2(\omega)} \Bigr)
}} \, .
\end{equation}
Thus the kernel is of the form 
\begin{equation}
  \label{eq:c47}
k(x,y) = \mF(\overline{T} c_{\Lambda^{1/2}}) (y-x) + \mF( T
c_{\Lambda^{1/2}}) (x-y) \, ,
\end{equation}
and again there exists a $b_2$ such that $\int _{|y-x| \geq b_2}
|k(x,y)|^2 \, dx < \epsilon ^2$. 

By combination of these cases and adjusting the choice of $b$, $b=
\max (b_0,b_1,b_2)$,  the
statement follows when $y>a$. 

 \textbf{Case III: $\pmb{ y < - a}$.} This case is treated in complete
 analogy to the case $y>a$ by using the remaining scattering relations
 \eqref{eq:fundsols}. 
\end{proof}

\begin{rem} \label{rem:alert}
  At this point we must alert the reader to the miracle happening
  in~\eqref{eq:5}.   The unitarity of the scattering matrix implies
  that the coefficient of $e^{i \omega (x+y)}$ vanishes. If this were
  not the case, we would have no control over the size of \\ $\int
  _{|x-y|\geq b_2} |k(x,y)|^2 \, dx$, and the whole proof would
  break down. It is this seemingly little detail that made us favor  the
  Schr\"odinger picture  over the Sturm-Liouville picture. 
\end{rem}

\begin{proof}[Proof of the homogenous approximation property ---  Lemma~\ref{lem:hap}]
We must show that there exists $b >0$  such that 
$\sum _{x\in X: \abs{x-y} >b}\abs{k(x,y)}^2    <
\epsilon ^2$ for all $y\in \br $. 
This is the discrete analogue of Lemma~\ref{lem:conthap}, and its
proof is  roughly parallel to the one of Lemma~\ref{lem:conthap}. 

\textbf{Case I: $\pmb{\abs y \leq a}$.}  As $X$ is a set of stable sampling,
the mapping $ f \mapsto \big(f(x)\big)_{x \in X}$ is \cont\ from $\PW
[B_q]\Lambda$ to $\ell^2(X)$, and thus maps compact sets in
$PW_\Lambda (B_q)$ to compact sets in $\ell ^2(X)$. Applying this
remark to the compact set $\{ k(\cdot , y) : |y|\leq a\}$ (as shown in
the 
proof of  Lemma~\ref{lem:conthap}, Case I), we see that
the set of samples $\{ (k(x,y)_{x\in X} : |y|\leq a\}$ is compact in
$\ell ^2(X)$.


The version of the Kolmogorov-Riesz theorem  for sequences  implies
that for every  $\epsilon>0$ there exists a number $b= b_\epsilon $ such that 
\[
\sup _{|y| \leq a} \sum_{\substack{x \in X\\\abs{x} > b}}\abs{k(x,y)}^2  < \epsilon \,,
\]
as claimed. 

\textbf{Case II: $\pmb{ y > a}$.} 

 We split the sum  into three 
parts: 
\[
\sum _{x\in X: \abs{x-y} >b}\abs{k(x,y)}^2 = \sum _{\substack{\abs{x-y} >b \\
    \abs x \leq a}}\abs{k(x,y)}^2  + \sum_{\substack{\abs{x-y} >b \\
     x > a}}\abs{k(x,y)}^2  +\sum_{\substack{\abs{x-y} >b \\
     x <- a}}\abs{k(x,y)}^2   = A + B+C\, . 
\]

\emph{Estimate of $A$.} 
We seek  $b_0$ large enough, such that
\[
\sum_{\substack{x \in X\\\abs{x} \leq a}}\abs{k(x,y)}^2 < \epsilon \qquad \text{for  } \quad \abs y > b_0> a
\]
  As in the proof of the parallel case of Lemma \ref{lem:conthap} we obtain
  \begin{align*}
 A & =  \sum_{\substack {x \in X\\ \abs x \leq a}}\abs{k(x,y)}^2 dx =
\sum_{i,k=1}^2 \iint_{{\Lambda^{1/2}} \times
  {\Lambda^{1/2}}}\Big(\sum_{\substack {x \in X\\\abs x \leq a}}
\Phi_i(\omega,x) \overline{\Phi_k(\mu,x)} \Big)\;
\overline{\Phi_i(\omega,y)} \Phi_k(\mu,y) \,d\mu \, d\omega \\
&= \sum_{i,k=1}^2 \iint_{{\Lambda^{1/2}} \times
  {\Lambda^{1/2}}} \Psi _{ik}(\omega,\mu ) \;
\overline{\Phi_i(\omega,y)} \Phi_k(\mu,y) \,d\mu \, d\omega \, ,
  \end{align*}
where we denote the inner sum by  $\Psi _{ik}(\omega,\mu ) = \sum_{\substack {x \in X\\\abs x \leq a}}
\Phi_i(\omega,x) \overline{\Phi_k(\mu,x)} $. Since $X$ is relatively
separated, $\# (X\cap [-a,a]) $ is finite. Furthermore, the set of 
eigenfunctions is uniformly bounded on the compact set $\Lambda ^{1/2}
\times [-a,a]$, therefore $\Psi _{ik}$ is continuous and bounded on
$\Lambda ^{1/2} \times \Lambda ^{1/2} $. 

Next, as in the proof of Lemma~\ref{lem:conthap}, the mapping $y\mapsto
\sum _{x\in X: |x| \leq a} |k(x,y)|^2$ is the two-dimensional Fourier
transform of a bounded continuous function, and by the
Riemann-Lebesgue Lemma we obtain $\lim _{y\to \infty } 
\sum _{x\in X: |x| \leq a} |k(x,y)|^2 = 0$. Thus $|A| < \epsilon $ for
all 
$y$ sufficiently   large.

\emph{Estimate of $B$ and $C$.} For the estimate of $B$ and $C$ we use
the formulas for the kernel  \eqref{eq:specfunas} (for $x>a$) and
\eqref{eq:c47} (for $x< -a$). In both cases the kernel is a sum of
Fourier transform of the scattering coefficients restricted to
$\Lambda ^{1/2}$, i.e., of $c_{\Lambda ^{1/2}}, Tc_{\Lambda ^{1/2}},
R_1 c_{\Lambda ^{1/2}}$, and $R_2c_{\Lambda ^{1/2}}$.  
Since $\Lambda ^{1/2}$ is assumed to be bounded, the function $x
\mapsto k(x,y)$ is thus the restriction of a classical \bdl\ function
to one of the intervals $[a,\infty )$ or $(-\infty ,-a]$. 

We use a
local version of  the Plancherel-Polya-Theorem from
~\cite[Lemma~1]{Gr96landau}. We set $f^\sharp (x) = \sup _{|y-x|\leq
  1} |f(y)|$ and note that if $f\in \Ltwo $ is \bdl\ with $\supp \hat{f}
\subseteq [\Omega , \Omega ]$, say, then $f^\sharp \in \Ltwo $. If $X$
is separated with separation $\min _{x,x'\in X} |x-x'| \geq 1$, then 
\begin{equation}
  \label{eq:c56}
\sum _{x\in X, |x| \geq b} |f (x)|^2\leq \int _{|x| \geq b-1} |f^\sharp (x)|^2
\, dx \, .
\end{equation}
If $X$ is relatively separated, then this inequality holds with the
constant $n_0 = \max 
_{c\in \br } \# (X \cap [c,c+1]) $ on the right hand side. We now
apply ~\eqref{eq:c56} to the functions $x \mapsto \mF c  _{\Lambda
  ^{1/2}} (x-y)$ and $x \mapsto \mF (R_2 c  _{\Lambda
  ^{1/2}}) (x+y)$ and obtain 
$$
 \sum_{\substack{x\in X, \abs{x-y} >b  \\     x > a}}\abs{\mF c  _{\Lambda
     ^{1/2}} (x-y) }^2  \leq n_0
 \int _{|x-y| \geq b-1} |(\mF c  _{\Lambda
  ^{1/2}})^\sharp  (x-y)|^2 \, dx =  n_0
 \int _{|z| \geq b-1} |(\mF c  _{\Lambda
  ^{1/2}})^\sharp  (z)|^2 \, dz \, .
$$
and for $y>a$ also
$$
 \sum_{\substack{x\in X: \abs{x-y} >b   \\   x > a}}\abs{\mF (R_2 c 
   _{\Lambda ^{1/2}})(x-y)}^2  \leq n_0
 \int _{|z | \geq b-1} \abs{\mF (R_2 c 
   _{\Lambda ^{1/2}})^\sharp (z)}^2 \, dz \, .
$$
Consequently, 
$$
B= \sum_{\substack{\abs{x-y} >b \\
     x > a}}\abs{k(x,y)}^2  \leq 4 n_0   \int _{|z| \geq b-1} \Big( |(\mF c  _{\Lambda
  ^{1/2}})^\sharp  (z)|^2 + |(\mF  (R_2 c  _{\Lambda
  ^{1/2}})^\sharp  (z)|^2\Big) \, dz < \epsilon \, ,
$$
for  $b$
large enough. 
Likewise, for large $b$
$$
C = \sum_{\substack{\abs{x-y} >b \\
     x < -a}}\abs{k(x,y)}^2  \leq 4 n_0  \int _{|z| \geq b-1}  |(\mF (T c  _{\Lambda
  ^{1/2}})^\sharp  (z)|^2  \, dz < \epsilon
$$

\textbf{Case III: $\pmb{ y <- a}$.}
This case is symmetric to Case II and settled with the same argument. 
\end{proof}

\section{Summary and  Outlook }

In this work we have argued that the spectral subspaces of a
Sturm-Liouville operator $f \mapsto - (pf')'$ on $L^2(\br )$  with a
positive parametrizing function $p$ 
may serve as
a model for functions of variable bandwidth.  
 Our results strongly
support the intuition that the quantity $p(x)^{-1/2}$ is a measure for
the local \bw\ of such a function. This intuition is backed up by sampling
theorems (with only minimal assumptions on $p$), and by necessary
density conditions for sampling and for interpolation (for the model
case of eventually constant $p$). 

Clearly  the project of variable \bw\ has a much bigger scope than can
be treated in a single paper. The notion of variable \bw\ requires
much more and deeper investigations and obviously raises a multitude  of new
questions both in  sampling theory and also  about the fine spectral
properties  of Sturm-Liouville operators. Let us sketch a few
directions (some of which we plan to address in subsequent
publications, and some of which we have no concrete ideas about). 

(a) Clearly the model case of an eventually constant \bw\
parametrization $p$ is quite restrictive. It seems that a version of  Theorem~\ref{ncdschr}
can be proved under the assumption that $p$ is \emph{asymptotically
  constant}, i.e.,   $|p(x)^{-1} - p_{\pm}^{-1}| = \mO (|x|^{-\alpha })$ as $x\to
\pm \infty  $ for some
$\alpha >1$. This case, however, requires much more spectral theory of
Sturm Liouville operators, which according to~\cite{Weidmann03} is ``decidedly
more complicated''.

(b) In view of the classical results of Beurling~\cite{beur89} one may
conjecture that the density condition $D_p^-(X) > \Omega ^{1/2} / \pi
$ is sufficient for $X$ to be a set of sampling for $PW_{[0,\Omega
  ]}(A_p)$, at least for reasonable $p$. At this time it is not clear
how to replace   the maximum gap condition in Theorem~\ref{maxgap}  by the
average density of Beurling, because the functions in $PW_{[0,\Omega
  ]}(A_p)$ are no longer entire. 

(c) In the special case $p(x) = p_-$ for $x\leq 0$ and $p(x) = p_+$
for $x>0$ we have found a set of sampling and interpolation,
equivalently, an orthonormal basis of reproducing kernels. Is there a
Riesz basis of reproducing kernels in $PW_\Lambda (A_p)$ for more
general parametrizing functions $p$? This problem is hard even for $p
\equiv 1$ and disconnected spectral sets $\Lambda $. See~\cite{KN15} for a
recent breakthrough.

(d) Spectral perturbation theory: How are the Paley-Wiener spaces
$PW_\Lambda (A_{p_1})$ and $PW_\Lambda (A_{p_2})$ related when $p_1$
and $p_2$ are close in some sense? 

(e) Clearly all questions may be treated in the multivariate setting
by considering a strongly elliptic second order differential operator
$f \to - \nabla (M \nabla f)$ for some matrix-valued function $x\to
M(x)$. In this case only  the 
existence of frames of reproducing kernels is known  from the general  work of
Pesenson and Zayed~\cite{Pes01,Pesenson09} (by sampling densely enough), but all quantitative
questions about necessary and sufficient conditions for sampling 
are wide open.  Likewise, the connection of the spectral subspaces to
a local \bw\ is far from transparent. In higher dimensions we expect
the geometry associated to elliptic differential operators to play a
more visible and prominent role. In Sections~5 and~6  the explicit metric
$d(y,z) = | \int _y ^z p(x)^{-1/2} dx|$ played an important role, in
higher dimensions  analogous metrics are known as
Carnot-Carath\'eodory metrics, see, e.g., ~\cite{Nagel86,Nagel03}. We expect these to appear  in the
correct definition of a Beurling density and in the formulation of
sampling results. 

(f) Computational aspects: the ultimate goal would be the use of
variable \bw\ for adaptive signal reconstruction. The idea is choose
the  local \bw\  according to the local sampling density and then
reconstruct a function in a space of variable \bw . 
  Given a nonuniform
sampling set $X = \{x_j\}$ and samples $y_j = f(x_j)$, we would like
to proceed as follows: (i)  find  a
\bw\ parametrizing function $p$ such that the maximum gap
condition~\eqref{eq:48} is satisfied. (ii) Construct a function in
$PW_{[0,\Omega ]}(A_p)$ from the data $(X, f(X))$ by means of an
algorithm based on Theorem~\ref{thm:rec-alg}. This may seem hopeless
for general parametrizing functions $p$, because any procedure
requires the knowledge of the reproducing kernel.  The discussion of
Section~4 shows that at least for 
piecewise constant $p$ the sampling theory can be made more
explicit. Therefore this 
idea carries some potential for the numerical realization.


 \appendix %

 \section{Computation of the spectral measure}
 \label{sec:comp-spect-meas}

\subsection*{Spectral measure}
\label{sec:spectral-measure}

For the explicit construction of the spectral measure $\rho$ assume
$\tau$ is  LP  at $\pm \infty$, and denote   the unique
solution of  
$(\tau-z)\phi=0$,  $z \notin  \br$,
that lie left (right) in $L^2(\br)$ by  $\phi_-(z,\cdot )$ (by $ \phi_+(z,
\cdot )$). 
Then the resolvent  $R_z(A)= \inv{(A-z)}$  of the \saj\ realization
$A$ of $\tau$ is the  integral operator
\begin{equation}
  \label{eq:41}
  R_z(A) g(x) 
= \frac{1}{W_p(\phi_+,\phi_-)(z)} 
 \Big( \phi_+(z,x)\int_{-\infty}^x \phi_-(z,u) g(u) \,du + \phi_-(z,x)\int_x^\infty \phi_+(z,u) g(u) \, du \Big) \,,
\end{equation}
where $W_p(f,g)(z)= (pf'\, g - f\, pg')(z,x)$ is the generalized
Wronski determinant. Note that $W_p(\phi _+,\phi _-)$ is independent
of $x$  for the  solutions of $(\tau-z)\phi=0$
~\cite[13.21]{Weidmann03},\cite[Eq.\ (9.6)]{Teschl09}.
 
 Assume that  the components of $\Phi(z,\cdot)$ form a fundamental
 system of solutions of $(\tau-z)\phi=0$ that depend  \cont ly on $z$
 in a complex neighborhood $Q$ of the interval $(\alpha,\beta)
 \subseteq \br$. Then  there exist  $2 \times 2$ matrices $m^{\pm}(z)$
 for $z \in Q \cap (\bc \setminus \sigma(A))$, such that the integral kernel
 $r_z$ of the resolvent $R_z(A)$ can be written as 
\begin{equation}
   \label{eq:43}
   r_z(x,y)=
   \begin{cases}
    \overline{\Phi(\bar z,x)}\cdot  m^+(z)\Phi(z,y) \,,\quad &y \leq x \,, \\
      \overline{\Phi(\bar z,x)}\cdot  m^-(z)\Phi(z,y) \,,\quad &y > x \, .
   \end{cases}
 \end{equation}
 For an interval $ (\gamma,\lambda] \subseteq (\alpha,\beta)$  the spectral measure is given by the\emph{
   Weyl-Titchmarsh-Kodaira formula} 
\begin{equation}
  \label{eq:44}
  \rho\big( (\gamma,\lambda]\big)= \frac{1}{2 \pi i} \lim_{\delta
    \searrow 0}\lim_{\epsilon \searrow 0}
  \int_{\gamma+\delta}^{\lambda+\delta} \big( m^\pm(t+ i \epsilon)
  -m^\pm(t- i \epsilon) \big) dt\, . 
\end{equation}
See e.g \cite[14.5]{Weidmann03},\cite[XII, 5.18]{DS2} for a proof.

We now  compute the spectral measure for $A_p$ studied in Section \ref{sec:model-case}.
 
 Set $\phi_1=\phi_+$, $\phi_2=\phi_-$ and
\begin{align*}
  a_1=\frac{1}{2} \Big(1+ \sqrt{\frac{p_+}{p_-}}\Big),\quad & b_1=\frac{1}{2} \Big(1- \sqrt{\frac{p_+}{p_-}}\Big)\\
  a_2=\frac{1}{2} \Big(1+ \sqrt{\frac{p_-}{p_+}}\Big),\quad &
  b_2=\frac{1}{2} \Big(1- \sqrt{\frac{p_-}{p_+}}\Big) \, .
\end{align*}
 Then 
\[
W_p(\phi_1,\phi_1)=- i \sqrt z (\pp+\pmi) \, .
\]
For $ y \leq x$ the resolvent kernel can be written as
\begin{align*}
    r_z (x,y) &=\frac{1}{W_p(\phi_1,\phi_2)}\phi_1(z,x) \phi_2(z,y) \\
              &=\sum_{j,k=1}^2m_{jk}^+\overline{\phi_j(\bar z,x))}\phi_k(z,y)
\end{align*}
If $x>0$,
\[
\phi_1(z,x)= \frac 1 {\overline{a_2}}\Big( \overline {\phi_2(\bar z,x)} - \overline{b_2} \,\overline{\phi_1(\bar z,x)} \Big)\,, 
\]
so
\[
r_z(x,y)=\frac{1}{W_p(\phi_1,\phi_2)\,\overline{a_2}}
\Big(
- \overline{b_2} \,\overline{\phi_1(\bar z,x)}  \phi_2(z,y)+ \overline {\phi_2(\bar z,x)} \phi_2(z,y)
\Big)
\]
which yields for the matrix $m^+$
\[
m^+(z)=\frac{1}{W_p(\phi_1,\phi_2)\,\overline{a_2}}
\begin{pmatrix}
  0 & * \\
0  & 1
\end{pmatrix}
\]
By a similar calculation for $x < 0$,
\[
m^+(\bar z)=\frac{1}{\overline{W_p(\phi_1,\phi_2)}\,{a_1}}
\begin{pmatrix}
  1 & ** \\
0  & 0
\end{pmatrix} \, ,
\]
and so
\[
m^+(\lambda+ 0 i)-m^+(\lambda - 0 i)= \frac{i}{\sqrt\lambda (\pp +\pmi)^2}
\begin{pmatrix}
  2 \pmi & *** \\
0  &  2 \pp
\end{pmatrix} \, .
\]

\section{Time Warping}
\label{sec:time-warping}
We discuss briefly how time-warping is related to our approach with
spectral subspaces. Let $p$ be a parametrizing function with $0<c \leq
p(x) \leq C <\infty $ and consider the differential expression $f \to
-i p f'$. By choosing the correct measure on $\br $ and an appropriate
domain, we obtain the self-adjoint operator
\[
B_p = -i p \tfrac{d}{dx} \,,\quad \mD(B_p)=\set{ f \in L^2(\br, \frac{dx}{p(x)})
  \colon B_p f \in  L^2(\br , \frac{dx}{p(x)})} 
\]
on  $L^2(\br, \frac{dx}{p(x)})$ with corresponding spectral
projections $c_\Lambda (B_p)$ for $\Lambda \subseteq \br $. Thus a
function $f\in L^2(\br, dx/p(x))$ is $B_p$-\bdl\ with spectral set
$\Lambda $, in short $f\in PW_\Lambda (B_p)$,  if 
$ f= c_\Lambda (B_p) f  $. 
In this case \bdl\ functions possess the following explicit
description.

\begin{prop}
  Set $\eta (x) = \int _0 ^x \tfrac{1}{p(t)}\, dt $. Then $f\in PW
  _\Lambda (B_p)$, \fif\ there exists $F\in L^2(\br )$ with
  $\mathrm{supp}\, F \subseteq \Lambda $, such that 
  \begin{equation}
    \label{eq:c33}
    f(x) = \int _\Lambda F(\lambda ) e^{i\lambda \eta ^{-1}(x)} \,
    d\lambda = \big(\mF^{-1} F\big) (\eta ^{-1}(x)) \, .
  \end{equation}
Thus $f$ is obtained from the \bdl\ function $\mF^{-1} F$ by
time-warping with $\inv \eta $.  
\end{prop}

\begin{proof}
  The proof follows easily, once we  have identified the spectral
  measure and diagonalized  $B_p$. The eigenfunctions $ - i p(x)
  \frac{d}{dx}\Phi(\lambda,x) = \lambda \Phi(\lambda,x)$ are given explicitly as 
\[
  \Phi(\lambda,x)=e^{i \lambda \int_0^x \frac {dt}{ p (t)} } = e^{i
    \lambda \eta(x)} \, , 
\]
and the corresponding spectral transform ist given by 
$$
 U_p f(\lambda)
=\int_\br f(x) \overline{\Phi(\lambda,x)} \tfrac{dx}{p(x)} = \int_\br
f(x) e^{-i \lambda \eta (x)}  \tfrac{dx}{p(x)}   \, .
$$
Using the substitution $y=\eta (x), dy = \eta'(x) dx =
\frac{dx}{p(x)}$, we find that 
$$
U_pf(x) = \int_\br
f(\eta ^{-1}(y))  e^{-i \lambda y } \, dy  = \mF( f \circ \eta ^{-1})
(\lambda ) \, .
$$
It is now easy to verify that $U_p$ is unitary from $L^2(\br,
\frac{dx}{p(x)})$  onto $\Ltwo $ and that $U_p$ diagonalizes $B_p$,
i.e., $ U_p B_p U_p ^* F(\lambda ) = \lambda F(\lambda )$.  
 The inverse $U_p ^{-1} = U_p^* : \Ltwo \to L^2(\br, \frac{dx}{p(x)})$ 
 is then given by 
$$
U_p^* F(x)  = \int _{\br } F(\lambda ) \Phi (\lambda , x) \, d\lambda
=  \int _{\br } F(\lambda ) e^{i\lambda \eta (x)} \, d\lambda \, .
$$
or 
$ U_p^* F =  (\mF^{-1} F) \circ \eta $. 
The spectral projection of $\Lambda $ is then $c_\Lambda (B_p)f  =
U^*_p c_\Lambda  U_p f $. Consequently, every $f\in PW_\Lambda (B_p)$
is given by 
$$
f(x)  =  \int _{\Lambda  } F(\lambda ) e^{i\lambda \eta (x)} \,
d\lambda  = (\mF^{-1} F)( \eta (x)) \, 
$$ for some $F \in L^2(\Lambda )$, as claimed.  
\end{proof}

\def\cprime{$'$} \def\cprime{$'$} \def\cprime{$'$}


\end{document}